\documentclass[12pt]{amsart}
\usepackage{amssymb}
\usepackage{amsthm,amsfonts}
\usepackage{amsfonts}
\usepackage{amsmath}
\usepackage[colorlinks]{hyperref}
\usepackage{enumitem}
\usepackage{graphicx}

\usepackage{tikz}
\usepackage{pgfplots}
\usepgfplotslibrary{fillbetween}
\usetikzlibrary{arrows.meta,calc}
\pgfplotsset{compat=1.17}
\usepackage[initials]{amsrefs}
\usepackage{color}
\RequirePackage{transparent}
\usepackage{geometry}
\definecolor{regionR}{RGB}{191,219,254}
\definecolor{regionA}{RGB}{249,199,79}
\definecolor{regionB}{RGB}{43,179,165}
\definecolor{regionC}{RGB}{167,139,203}
\definecolor{regionCfour}{RGB}{224,52,120}
\definecolor{regionD}{RGB}{238,240,243}

\definecolor{figA}{RGB}{189,214,239}      
\definecolor{figB}{RGB}{244,192,143}      
\definecolor{figCOne}{RGB}{204,185,227}   
\definecolor{figCTwo}{RGB}{160,211,204}   
\definecolor{figCThree}{RGB}{241,210,112} 
\definecolor{figCFour}{RGB}{235,145,135}  
\definecolor{figD}{RGB}{248,248,248}      
\definecolor{figZoom}{RGB}{142,38,53}     
\definecolor{figQ}{RGB}{31,78,121}        

\newtheorem{Theorem}{Theorem}[section]

\newtheorem{Lemma}[Theorem]{Lemma}

\newtheorem{proposition}[Theorem]{Proposition}

\theoremstyle{definition}

\newcommand{\R}{\mathbb{R}}

\DeclareMathOperator*{\ext}{ext}
\DeclareMathOperator*{\graph}{graph}
\begin{document}
	\title[Geometry of quadratic forms over a uniparametric family of squares]{Geometry of quadratic forms over a uniparametric family of squares}
	\author[Llorente]{J. Llorente}
	\address[J. Llorente]{\mbox{}\newline \indent Instituto de Matem\'atica Interdisciplinar (IMI)\newline \indent Departamento de An\'{a}lisis Matem\'{a}tico y Matem\'atica Aplicada \newline \indent
		Facultad de Ciencias Matem\'{a}ticas\newline \indent
		Plaza de Ciencias 3 \newline \indent
		Universidad Complutense de Madrid \newline \indent
		Madrid, 28040, Spain.}
	\email{jesllore@ucm.es}
	\author[Muñoz]{G. A. Muñoz-Fernández}
	\address[G. A. Muñoz-Fernández]{\mbox{}\newline \indent Instituto de Matem\'atica Interdisciplinar (IMI)\newline \indent Departamento de An\'{a}lisis Matem\'{a}tico y Matem\'atica Aplicada \newline \indent
		Facultad de Ciencias Matem\'{a}ticas\newline \indent
		Plaza de Ciencias 3 \newline \indent
		Universidad Complutense de Madrid \newline \indent
		Madrid, 28040, Spain.}
	\email{gustavo@ucm.es}
	\author[Rodríguez]{D. L. Rodr\'iguez-Vidanes}
	\address[D. L.~Rodr\'iguez-Vidanes]{\mbox{}\newline \indent Grupo de Análisis Matemático y Aplicaciones \newline \indent Departamento de Matemática Aplicada a la Ingeniería Industrial \newline \indent
		Escuela Técnica Superior de Ingeniería y Diseño Industrial\newline \indent
		Ronda de Valencia 3 \newline \indent
		Universidad Politécnica de Madrid \newline \indent
		Madrid, 28012, Spain.}
	\email{dl.rodriguez.vidanes@upm.es}
	\author[Seoane]{J. B.~Seoane-Sep\'{u}lveda}
	\address[J. B.~Seoane-Sep\'{u}lveda]{\mbox{}\newline \indent Instituto de Matem\'atica Interdisciplinar (IMI)\newline \indent Departamento de An\'{a}lisis Matem\'{a}tico y Matem\'atica Aplicada \newline \indent
		Facultad de Ciencias Matem\'{a}ticas\newline \indent
		Plaza de Ciencias 3 \newline \indent
		Universidad Complutense de Madrid \newline \indent
		Madrid, 28040, Spain.}
	\email{jseoane@mat.ucm.es}
	\author[Tag]{Hyung-Joon Tag}
	\address[Hyung-Joon Tag]{\mbox{}\newline \indent Department of Mathematics and Statistics
		\newline \indent University of North Carolina at Greensboro \newline \indent
		116 Petty Building, \newline \indent
		Greensboro, North Carolina, 27402, USA.}
	\email{hjtag4@gmail.com}
	\thanks{
		This work was partially carried out during two visits by the fifth author to the Department of Mathematical Analysis and Applied Mathematics, Faculty of Mathematical Sciences, Universidad Complutense de Madrid, Spain. The fifth author gratefully acknowledges the support and hospitality received during these visits.
	}
	\maketitle
\begin{abstract}
		We characterize the extreme points of the parametric space of quadratic forms on ${\mathbb R}^2$ endowed with the sup norm on the squares $\square_h := [h,h+1] \times [0,1]$ with $h>0$.
\end{abstract}

\section{Introduction}

The search for characterizations of the extreme points of the unit ball of a polynomial space of low dimension has  motivated a plethora of publications in the last 30 years. A few representative examples are given below:

\begin{enumerate}

	\item The spaces ${\mathcal P}(^2\ell_p^2({\mathbb R}))$ of the quadratic forms on ${\mathbb R}^2$ with the supremum norm on the unit ball of $\ell^2_p$  (see \cite{CK1,CK2,CK3,ChKiKi} for $p=1,2,\infty$ and \cite{G3} for $p\in(1,2)\cup(1,\infty)$).

	\item The space ${\mathcal P}(^3H)$ of the homogeneous polynomials of degree $3$ on a Hilbert space $H$ endowed with the supremum norm over the closed unit ball of $H$ (see \cite{G5}).
	
	\item The spaces ${\mathcal P}_2({\mathbb R})$ and ${\mathcal P}_2({\mathbb C})$ of real quadratic polynomials with the supremum norm on $[-1,1]$ and the closed unit disk in ${\mathbb C}$, denoted by ${\mathbb D}$, respectively (see \cite{AK}).	
	
	\item The space ${\mathcal P}_{m,n}({\mathbb R})$ of real trinomials of the form $ax^m+bx^n+c$ with $m>n$ ($m,n\in{\mathbb N}$) endowed with the supremum norm on $[-1,1]$ (see \cite{MS}).
	
	\item The space ${\mathcal P}_{m,n,k}({\mathbb C})$ of the real trinomials of the form $az^m+bz^n+cz^k$ with $a,b,c\in{\mathbb R}$ and $m,n,k\in {\mathbb N}$ with $m>n>k$ endowed with the supremum norm on ${\mathbb D}$ (see \cite{N}).
	
	\item The space ${\mathcal P}_2(\Delta)$ of polynomials of degree at most $2$ on ${\mathbb R}^2$ endowed with the	supremum norm over the simplex $\Delta$ of vertices $(0,0)$, $(1,0)$ and $(0,1)$ (see \cite{MN1,MN2,MN3}).	
	
	\item The space ${\mathcal P}(^2\Delta(P,Q))$ of the quadratic forms on ${\mathbb R}^2$ endowed with the supremum norm on the simplex $\Delta(P,Q)$ with vertices at the origin $(0,0)$ and the points $P=(\alpha_1,\alpha_2)$ and $Q=(\beta_1,\beta_2)$ with $P$ and $Q$ being linearly independent (see \cite{AGM2025}). The particular case where $P=(1,0)$ and $Q=(0,1)$ was studied in (see \cite{MRS}).
	
	\item The space ${\mathcal P}(^2{\mathcal O}_w^2)$ of the quadratic forms on ${\mathbb R}^2$ endowed with the norm
	$$
	\|ax^2+bxy+cy^2\|_{{\mathcal O}_w^2}=\sup\left\{|ax^2+bxy+cy^2|:\|(x,y)\|_{\text{oct}(w)}\leq 1\right\}
	$$
	where 
	$$
	\|(x,y)\|_{\text{oct}(w)}=\max\left\{|x|,|y|,\frac{|x|+|y|}{1+w}\right\}
	$$
	for a fixed $w\in[0,1]$ (see \cite{Ki2}). 
	
	\item The space ${\mathcal P}(^2\Box)$ of the quadratic forms $ax^2+bxy+cy^2$ with $a,b,c\in{\mathbb R}$ endowed with the supremum norm on the square $\Box=[0,1]^2$ (see \cite{GMSS}).
	
	\item The space ${\mathcal P}(^2{\mathcal H}_w^2)$ of the quadratic forms $ax^2+bxy+cy^2$ with $a,b,c\in{\mathbb R}$ endowed with the norm $\|ax^2+bxy+cy^2\|_{{\mathcal H}_w^2}=\sup\left\{|ax^2+bxy+cy^2|:\|(x,y)\|_{\text{hex}(w)}\leq 1\right\}$ where $\|(x,y)\|_{\text{hex}(w)}=\max\left\{|y|,(|x|+(1-w)|y|)\right\}$ for a fixed $w\in[0,1]$ (see \cite{Ki3}). 
	
	\item The space ${\mathcal P}(^2D(\alpha,\beta))$ (see \cite{MPSW,BMRS}) of the quadratic forms $ax^2+bxy+cy^2$ with $a,b,c\in{\mathbb R}$, $\alpha,\beta\in[0,\infty)$ and $\alpha\leq \beta$, endowed with the supremum norm on the circular sectors $D(\alpha,\beta)=\{re^{i\theta}:r\in[0,1],\theta\in[\alpha,\beta]\}$.
	
	\item The space ${\mathcal P}_{m,n}^h({\mathbb R})$ of the homogeneous trinomials of the form $ax^m+bx^{m-n}y^n+cy^m$ with $a,b,c\in{\mathbb R}$ and $m\ge n$ ($m,n\in{\mathbb N}$) endowed with the supremum norm on the square $[-1,1]^2$ (see \cite{JMR2021,GMS2023,GJMMS2025}).

\end{enumerate}

An interested reader can find in monograph \cite{FGMMRS} an extensive compilation of results on the geometry of polynomial spaces, including most of the cases mentioned in the list above.

By combining the Krein--Milman theorem with the aforementioned geometric studies, one can derive several sharp polynomial inequalities using the so-called Krein--Milman approach. Numerous examples demonstrating the successful application of this approach to obtain optimal polynomial inequalities can be found here \cite{AJMS,AJMS2,AMRS,AEMRS,DMPS,JMMS,JMPS,MPS,MRS,MSanSeo1,MSanSeo2,MSS,MN4}. Another comprehensive monograph  focusing on the application of the previously described geometric results to the derivation of exact polynomial inequalities is available in \cite{GJMMS2025}.

In \cite{AGM2025}, the authors investigate the geometry of the unit ball of a parametric family of polynomial spaces defined over the triangles $\Delta_h$ with vertices $(0,0)$, $(1,h)$, and $(h,0)$. Motivated by this work, in the present paper we further explore the geometric properties of one-parameter families of polynomial spaces. In particular, for $h>0$, let
\[
\square_h := [h,h+1]\times[0,1]\subseteq\mathbb{R}^2
\]
denote the translation of the unit square $[0,1]^2$ by a distance $h$ in the positive direction of the $x$-axis.
Let $\mathcal P(^2 \square_h)$ be the vector space of $2$-homogeneous polynomials on $\R^2$ endowed with the norm
$$
\|P\|_h := \sup \{ |P(x,y)| : (x,y)\in \square_h \}.
$$
Clearly, we can express $(\mathcal P(^2 \square_h),\|\cdot\|_h)$ as the three-dimensional space $(\R^3,\|\cdot\|_h)$ where
$$
\|(a,b,c)\|_h := \|ax^2+bxy+cy^2\|_h.
$$
We denote the closed unit ball and the unit sphere of ${\mathcal P}(^2\square_h)$ by 
\begin{align*}
	{\mathsf B}_h & := \{(a,b,c)\in{\mathbb R}^3:\|(a,b,c)\|_h\leq 1\},\\
	{\mathsf S}_h & := \{(a,b,c)\in{\mathbb R}^3:\|(a,b,c)\|_h= 1\},
\end{align*}
respectively. The $ac$-plane, i.e., $b=0$, is denoted by $\pi_{ac}$ and the orthogonal projection of ${\mathsf{B}_h}$ onto the $ac$-plane by $\pi_{ac}({\mathsf B}_h)$. Similar notations can be adopted for the other planes and the orthogonal projections onto those planes.   

The space ${\mathcal P}(^2\square_h)$ will be studied according to the scheme used in~\cite{AK}. First we will obtain an explicit formula for $\|(a,b,c)\|_h$ in Section~\ref{sec:formula}. That formula will be crucial in order to project the unit ball ${\mathsf B}_h$ onto the $ac$-plane, which will be done in Section~\ref{sec:Projection}. The projection $\pi_{ac}({\mathsf B}_h)$ will be needed in Section~\ref{sec:Parametriztion} to obtain a parametrization of the unit sphere ${\mathsf S}_h$. Once we parametrize ${\mathsf S}_h$, the extreme points of ${\mathsf B}_h$ can be easily characterized, as we will show in Section~\ref{sec:ExtremePoints}.


\section{A formula for the norm}\label{sec:formula}

The aim of this section is to derive an explicit formula for the norm $\|\cdot\|_h$.

We first consider the case $b=0$.

\begin{proposition}\label{prop:bzero}
For every $a,c\in\mathbb R$,
\[
\|(a,0,c)\|_h
=
\max\bigl\{|ah^2|,|a(h+1)^2|,|ah^2+c|,|a(h+1)^2+c|\bigr\}.
\]
Consequently, if $a\neq0$, then
\[
\|(a,0,c)\|_h=
\begin{cases}
\operatorname{sgn}(a)(-ah^2-c),
& \text{if } \displaystyle \frac ca\le -h^2-(h+1)^2,\\[4pt]
|a|(h+1)^2,
& \text{if } \displaystyle -h^2-(h+1)^2\le\frac ca\le0,\\[4pt]
\operatorname{sgn}(a)\bigl(a(h+1)^2+c\bigr),
& \text{if } \displaystyle \frac ca\ge0,
\end{cases}
\]
whereas $\|(0,0,c)\|_h=|c|$ for every $c\in\mathbb R$.
\end{proposition}

\begin{proof}
Put $X=x^2$ and $Y=y^2$. Then
\[
(X,Y)\in[h^2,(h+1)^2]\times[0,1]
\]
and the map $(X,Y)\longmapsto aX+cY$ is affine. 
Hence, the maximum of its absolute value is attained at one of the four vertices of this rectangle. This gives the first formula, and the remaining assertions follow immediately.
\end{proof}

\begin{figure}[t]
\centering
\includegraphics[width=0.30\linewidth]{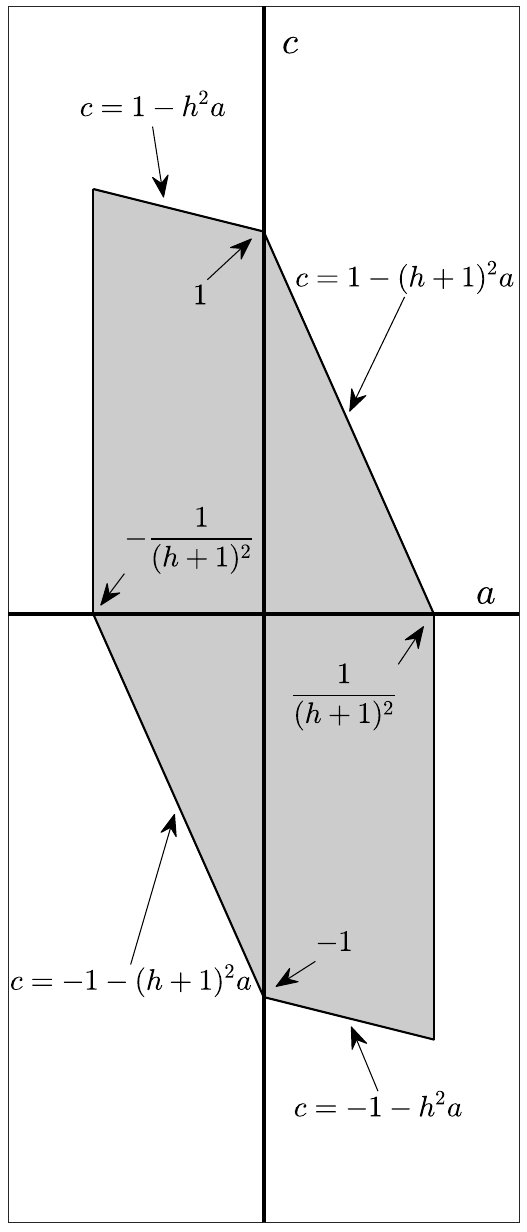}
\caption{Representation of ${\mathsf B}_h\cap\pi_{ac}$ for $h=1/2$.}
\label{fig:ac}
\end{figure}

Now assume that $b\neq0$. By homogeneity,
\begin{equation}\label{111}
\|(a,b,c)\|_h
=
|b|\left\|\left(\frac ab,1,\frac cb\right)\right\|_h.
\end{equation}
We therefore begin with the case $b=1$.

Let
\[
M := (\{h+1\}\times[0,1])\cup([h,h+1]\times\{1\}).
\]
For every $(x_0,y_0)\in\square_h\setminus M$, there exists $\lambda>1$ such that $\lambda(x_0,y_0)\in M$. Since every polynomial under consideration is $2$-homogeneous,
\[
|P(x_0,y_0)|
=
\lambda^{-2}|P(\lambda x_0,\lambda y_0)|
\le
|P(\lambda x_0,\lambda y_0)|.
\]
Thus, the norm is attained on $M$. The restrictions of $ax^2+xy+cy^2$ to the two components of $M$ are $ax^2+x+c$ and $a(h+1)^2+(h+1)y+cy^2$, respectively. Therefore,
\[
\|(a,1,c)\|_h
=
\max\bigl\{
\|(a,1,c)\|_{[h,h+1]},
\|(c,h+1,a(h+1)^2)\|_{[0,1]}
\bigr\},
\]
where, for $\alpha<\beta$,
\[
\|(A,B,C)\|_{[\alpha,\beta]}
:=
\sup_{x\in[\alpha,\beta]}|Ax^2+Bx+C|.
\]

We use the Aron--Klimek formula~\cite{AK}:
\begin{equation}\label{AKformula}
\|(A,B,C)\|_{[-1,1]}
=
\begin{cases}
\left|\dfrac{B^2}{4A}-C\right|,
& \text{if } |B|<2|A| \text{ and } \displaystyle \frac CA+1<\frac12\left(\left|\frac{B}{2A}\right|-1\right)^2,\\
|A+C|+|B|,
&\text{otherwise.}
\end{cases}
\end{equation}

For $\alpha<\beta$ and every $t\in [-1,1]$, let
\[
T(t)=\frac{\beta-\alpha}{2}t+\frac{\alpha+\beta}{2}.
\]
Then,
\[
\|(u,1,w)\|_{[\alpha,\beta]}
=
\left\|
\left(
\frac{u(\beta-\alpha)^2}{4},
\frac{u(\beta^2-\alpha^2)+\beta-\alpha}{2},
\frac{u(\alpha+\beta)^2}{4}+\frac{\alpha+\beta}{2}+w
\right)
\right\|_{[-1,1]}.
\]
In particular,
\[
\|(a,1,c)\|_{[h,h+1]}
=
\left\|
\left(
\frac a4,
\frac{(2h+1)a+1}{2},
\frac a4(2h+1)^2+\frac{2h+1}{2}+c
\right)
\right\|_{[-1,1]},
\]
and
\[
\|(c,h+1,a(h+1)^2)\|_{[0,1]}
=
\left\|
\left(
\frac c4,
\frac{c+h+1}{2},
 a(h+1)^2+\frac{h+1}{2}+\frac c4
\right)
\right\|_{[-1,1]}.
\]

For $(a,c)\in\mathbb R^2$, define
\begin{align*}
f(a,c)
&:=
\left\|
\left(
\frac a4,
\frac{(2h+1)a+1}{2},
\frac a4(2h+1)^2+\frac{2h+1}{2}+c
\right)
\right\|_{[-1,1]},\\
g(a,c)
&:=
\left\|
\left(
\frac c4,
\frac{c+h+1}{2},
 a(h+1)^2+\frac{h+1}{2}+\frac c4
\right)
\right\|_{[-1,1]}.
\end{align*}
Thus, $\|(a,1,c)\|_h=\max\{f(a,c),g(a,c)\}$.
Formula~\eqref{AKformula} gives
\begin{equation}\label{f}
f(a,c)=
\begin{cases}
\left|\dfrac1{4a}-c\right|,& \text{if } (a,c)\in\mathcal R,\\[7pt]
\left|c+h+1+a(h+1)^2-\dfrac{a(2h+1)+1}{2}\right|
+\dfrac12|a(2h+1)+1|,& \text{if } (a,c)\notin\mathcal R,
\end{cases}
\end{equation}
where
\begin{align}\label{R1}
\mathcal R
:={}&
\left\{(a,c)\in\mathbb R^2:
-\frac1{2h}<a\le-\frac1{2h+1},\,
c>\frac1{8a}-\frac{a(h+1)^2}{2}-\frac{h+1}{2}
\right\}
\nonumber\\
&\cup
\left\{(a,c)\in\mathbb R^2:
-\frac1{2h+1}\le a<-\frac1{2(h+1)},\,
c>\frac1{8a}-\frac{h(ha+1)}{2}
\right\},
\end{align}
and
\begin{equation}\label{g}
g(a,c)=
\begin{cases}
(h+1)^2\left|\dfrac1{4c}-a\right|,& \text{if } (a,c)\in\mathcal S,\\
\dfrac12|c+h+1+2a(h+1)^2|+\dfrac12|c+h+1|,& \text{if } (a,c)\notin\mathcal S,
\end{cases}
\end{equation}
where
\begin{align}\label{R2}
\mathcal S
:={}&
\left\{(a,c)\in\mathbb R^2:
-(h+1)\le c<-\frac{h+1}{2},\, a>\frac1{8c}
\right\}
\nonumber\\
&\cup
\left\{(a,c)\in\mathbb R^2:
c\le-(h+1),\,
 a>\frac1{8c}-\frac{c+h+1}{2(h+1)^2}
\right\}.
\end{align}

\begin{figure}[t]
\centering
\resizebox{0.92\textwidth}{!}{%
\begin{tikzpicture}
\begin{axis}[
width=15cm,height=8.4cm,
xmin=-2.35,xmax=.22,ymin=-1.15,ymax=3,
axis lines=middle,
axis line style={black,semithick},
xlabel={$a$},ylabel={$c$},
label style={font=\normalsize},
tick label style={font=\small},
xtick={-1,0},ytick={-1,0,1,2,3},
clip=false,
]
\addplot[name path=topone,draw=none,domain=-2:-.6666667] {3};
\addplot[name path=lowone,draw=none,domain=-2:-.6666667,samples=200] {1/(8*x)-25*x/32-5/8};
\addplot[fill=black!18,draw=none] fill between[of=topone and lowone];
\addplot[name path=toptwo,draw=none,domain=-.6666667:-.4] {3};
\addplot[name path=lowtwo,draw=none,domain=-.6666667:-.4,samples=120] {1/(8*x)-x/32-1/8};
\addplot[fill=black!18,draw=none] fill between[of=toptwo and lowtwo];
\addplot[very thick,dashed,domain=-2:-.6666667,samples=200] {1/(8*x)-25*x/32-5/8};
\addplot[very thick,dashed,domain=-.6666667:-.4,samples=120] {1/(8*x)-x/32-1/8};
\draw[very thick,dashed] (axis cs:-2,.875) -- (axis cs:-2,3);
\draw[very thick,dashed] (axis cs:-.4,-.425) -- (axis cs:-.4,3);
\draw[gray,densely dashed] (axis cs:-.6666667,-.292) -- (axis cs:-.6666667,.15);
\node[font=\Huge] at (axis cs:-1.15,1.85) {$\mathcal R$};
\node[font=\small,align=center,fill=white,fill opacity=.88,text opacity=1,inner sep=2pt] at (axis cs:-1.23,1.02)
{$c=\dfrac1{8a}-\dfrac{a(h+1)^2}{2}-\dfrac{h+1}{2}$};
\draw[-{Stealth[length=2.2mm]}] (axis cs:-1.24,.88) -- (axis cs:-1.43,.35);
\node[font=\small,align=center,fill=white,fill opacity=.88,text opacity=1,inner sep=2pt] at (axis cs:-.67,-.91)
{$c=\dfrac1{8a}-\dfrac{h(ha+1)}{2}$};
\draw[-{Stealth[length=2.2mm]}] (axis cs:-.67,-.76) -- (axis cs:-.53,-.47);
\end{axis}
\end{tikzpicture}%
}
\caption{Sketch of $\mathcal R$ for $h=1/4$.}
\label{fig:Region_R}
\end{figure}
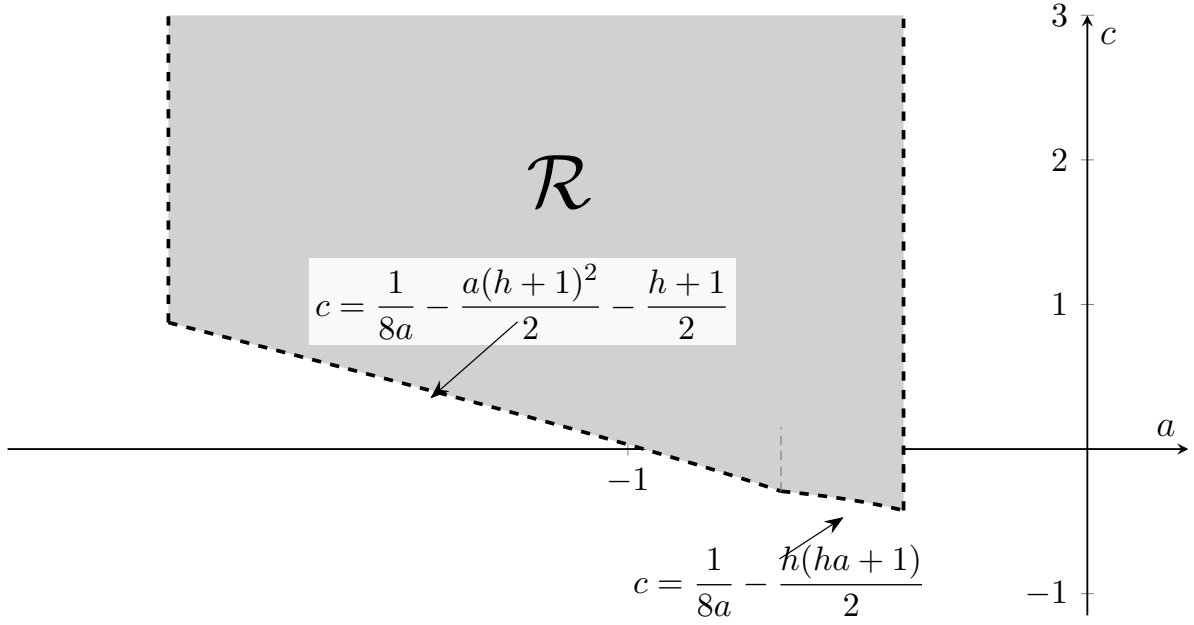

\begin{figure}[t]
\centering
\includegraphics[width=0.50\linewidth]{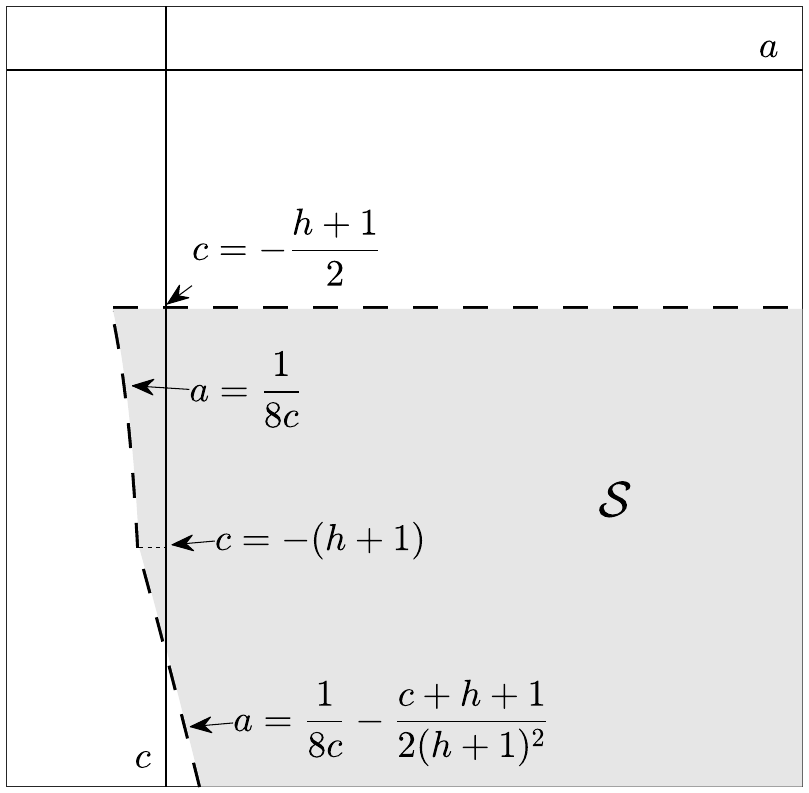}
\caption{Sketch of $\mathcal S$ for $h=1/4$.}
\label{fig:Region_S}
\end{figure}

Before comparing $f$ and $g$, we introduce the regions that occur in the final formula. Let
\[
\mathcal A
:=
\left\{(a,c)\in\mathbb R^2:
-\frac1{2h}<a<-\frac1{2(h+1)},\,
c+a(h+1)^2-\frac1{4a}\ge0
\right\}.
\]
Define $\mathcal B:=\mathcal B_1\cup\mathcal B_2$, where
\[
\mathcal B_1
:=
\left\{(a,c)\in\mathbb R^2:
-(h+1)\le c<-\dfrac{h+1}{2},\,
\displaystyle
a>\max\left\{\frac1{8c},
\frac{-c-h+\frac{(h+1)^2}{4c}}{h^2+(h+1)^2}
\right\}
\right\},
\]
\[
\mathcal B_2
:=
\left\{(a,c)\in\mathbb R^2:
 c\le-(h+1),\,
 a>\frac{-c-h+\frac{(h+1)^2}{4c}}{h^2+(h+1)^2}
\right\}.
\]
Next, let
\[
C_1
:=
\left\{(a,c)\in\mathbb R^2:
 a\le-\frac1{2h},\,
 c+a\left[h^2+(h+1)^2\right]+h\ge0
\right\},
\]
\[
C_2
:=
\left\{(a,c)\in\mathbb R^2:
 a\le-\frac1{2h+1},\, c\le-(h+1)
\right\},
\]
\[
C_3
:=
\left\{(a,c)\in\mathbb R^2:
\begin{array}{l}
\displaystyle a\ge-\frac1{2h+1},\, c+h-a(2h+1)\le0,\, c\le-\frac{h+1}{2},\,\\[3pt]
\displaystyle 
c^2+ch+ca\left[h^2+(h+1)^2\right]-\frac{(h+1)^2}{4}\ge0
\end{array}
\right\},
\]
and
\begin{equation}\label{C4}
C_4
:=
\left\{(a,c)\in\mathbb R^2:
\begin{array}{l}
\displaystyle c>-\frac{h+1}{2},\, c+h-a(2h+1)\le0,\\[3pt]
\displaystyle 2c+(2h+1)+a\left[h^2+(h+1)^2\right]\le0
\end{array}
\right\}.
\end{equation}
Finally, set
\[
\mathcal C:=C_1\cup C_2\cup C_3\cup C_4
\qquad \text{and} \qquad
\mathcal D:=(\mathcal A\cup\mathcal B\cup\mathcal C)^c.
\]

The component $C_4$ is nonempty exactly when
\[
0<h<\frac{\sqrt3-1}{2}=:h_0.
\]
Indeed, the inequalities in~\eqref{C4} admit a value of $c$ if and only if
\[
\frac{h-1}{2(2h+1)}<a<-\frac{h}{h^2+(h+1)^2},
\]
and this interval is nonempty precisely under the displayed condition.


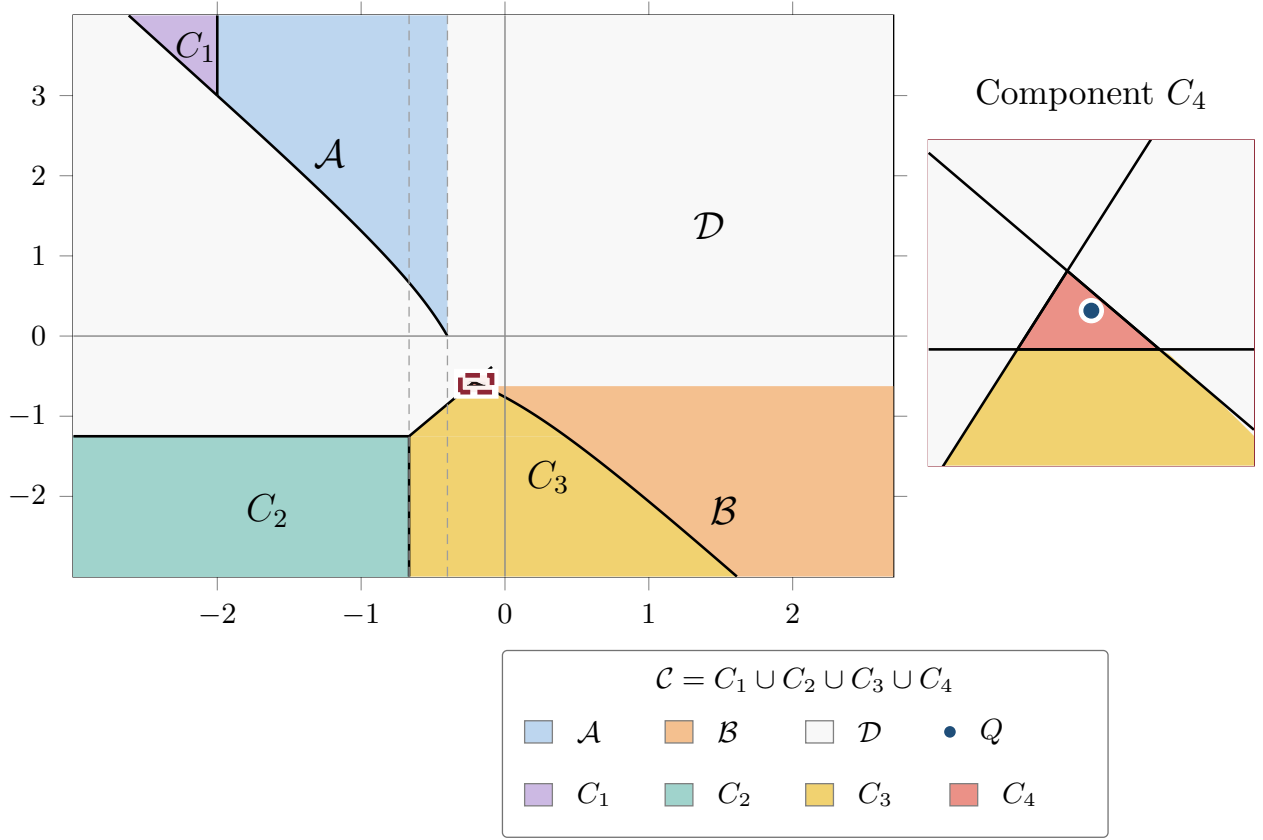
\begin{figure}[t]
	\centering
	\resizebox{0.98\textwidth}{!}{%
		\begin{tikzpicture}
			
			\begin{axis}[
				name=mainplot,
				at={(0,0)},
				anchor=south west,
				width=10.45cm,
				height=7.65cm,
				xmin=-3,xmax=2.70,
				ymin=-3,ymax=4,
				axis lines=box,
				xtick={-2,-1,0,1,2},
				ytick={-2,-1,0,1,2,3},
				tick align=outside,
				tick label style={font=\scriptsize},
				axis line style={black,semithick},
				clip=true,
				enlargelimits=false,
				]
				\addplot[draw=none,fill=figD] coordinates {
					(-3,-3) (2.70,-3) (2.70,4) (-3,4)
				} -- cycle;
				
				\addplot[draw=none,fill=figA] coordinates {
					(-2.00000000,4.00000000)
					(-0.40000000,4.00000000)
					(-0.40000000,0.00000000)
					(-0.42318841,0.07047846)
					(-0.44637681,0.13739883)
					(-0.46956522,0.20128824)
					(-0.49275362,0.26257460)
					(-0.51594203,0.32160886)
					(-0.53913043,0.37868163)
					(-0.56231884,0.43403556)
					(-0.58550725,0.48787487)
					(-0.60869565,0.54037267)
					(-0.63188406,0.59167664)
					(-0.65507246,0.64191356)
					(-0.67826087,0.69119287)
					(-0.70144928,0.73960953)
					(-0.72463768,0.78724638)
					(-0.74782609,0.83417594)
					(-0.77101449,0.88046202)
					(-0.79420290,0.92616101)
					(-0.81739130,0.97132285)
					(-0.84057971,1.01599200)
					(-0.86376812,1.06020815)
					(-0.88695652,1.10400682)
					(-0.91014493,1.14741992)
					(-0.93333333,1.19047619)
					(-0.95652174,1.23320158)
					(-0.97971014,1.27561959)
					(-1.00289855,1.31775153)
					(-1.02608696,1.35961680)
					(-1.04927536,1.40123309)
					(-1.07246377,1.44261653)
					(-1.09565217,1.48378192)
					(-1.11884058,1.52474281)
					(-1.14202899,1.56551166)
					(-1.16521739,1.60609994)
					(-1.18840580,1.64651820)
					(-1.21159420,1.68677623)
					(-1.23478261,1.72688304)
					(-1.25797101,1.76684699)
					(-1.28115942,1.80667585)
					(-1.30434783,1.84637681)
					(-1.32753623,1.88595659)
					(-1.35072464,1.92542141)
					(-1.37391304,1.96477711)
					(-1.39710145,2.00402911)
					(-1.42028986,2.04318249)
					(-1.44347826,2.08224201)
					(-1.46666667,2.12121212)
					(-1.48985507,2.16009699)
					(-1.51304348,2.19890055)
					(-1.53623188,2.23762647)
					(-1.55942029,2.27627822)
					(-1.58260870,2.31485905)
					(-1.60579710,2.35337205)
					(-1.62898551,2.39182010)
					(-1.65217391,2.43020595)
					(-1.67536232,2.46853217)
					(-1.69855072,2.50680121)
					(-1.72173913,2.54501537)
					(-1.74492754,2.58317685)
					(-1.76811594,2.62128772)
					(-1.79130435,2.65934994)
					(-1.81449275,2.69736537)
					(-1.83768116,2.73533580)
					(-1.86086957,2.77326290)
					(-1.88405797,2.81114827)
					(-1.90724638,2.84899344)
					(-1.93043478,2.88679984)
					(-1.95362319,2.92456887)
					(-1.97681159,2.96230184)
					(-2.00000000,3.00000000)
					(-2.00000000,3.00000000)
				} -- cycle;
				
				\addplot[draw=none,fill=figCOne] coordinates {
					(-34/13,4) (-2,3) (-2,4)
				} -- cycle;
				
				\addplot[draw=none,fill=figCTwo] coordinates {
					(-3,-3) (-2/3,-3) (-2/3,-5/4) (-3,-5/4)
				} -- cycle;
				
				
				\addplot[draw=none,fill=figCThree] coordinates {
					(-0.66666667,-3.00000000)
					(1.61217949,-3.00000000)
					(1.58662738,-2.96022727)
					(1.56104596,-2.92045455)
					(1.53543399,-2.88068182)
					(1.50979021,-2.84090909)
					(1.48411325,-2.80113636)
					(1.45840168,-2.76136364)
					(1.43265398,-2.72159091)
					(1.40686856,-2.68181818)
					(1.38104369,-2.64204545)
					(1.35517757,-2.60227273)
					(1.32926829,-2.56250000)
					(1.30331380,-2.52272727)
					(1.27731193,-2.48295455)
					(1.25126037,-2.44318182)
					(1.22515664,-2.40340909)
					(1.19899812,-2.36363636)
					(1.17278199,-2.32386364)
					(1.14650524,-2.28409091)
					(1.12016465,-2.24431818)
					(1.09375676,-2.20454545)
					(1.06727787,-2.16477273)
					(1.04072398,-2.12500000)
					(1.01409081,-2.08522727)
					(0.98737374,-2.04545455)
					(0.96056776,-2.00568182)
					(0.93366749,-1.96590909)
					(0.90666708,-1.92613636)
					(0.87956020,-1.88636364)
					(0.85233997,-1.84659091)
					(0.82499890,-1.80681818)
					(0.79752884,-1.76704545)
					(0.76992087,-1.72727273)
					(0.74216524,-1.68750000)
					(0.71425127,-1.64772727)
					(0.68616719,-1.60795455)
					(0.65790007,-1.56818182)
					(0.62943562,-1.52840909)
					(0.60075802,-1.48863636)
					(0.57184972,-1.44886364)
					(0.54269118,-1.40909091)
					(0.51326060,-1.36931818)
					(0.48353356,-1.32954545)
					(0.45348264,-1.28977273)
					(0.42307692,-1.25000000)
					(-0.66666667,-1.25000000)
				} -- cycle;
				
				\addplot[draw=none,fill=figCThree] coordinates {
					(-0.66666667,-1.25000000)
					(0.42307692,-1.25000000)
					(0.41212523,-1.23579545)
					(0.40112213,-1.22159091)
					(0.39006582,-1.20738636)
					(0.37895438,-1.19318182)
					(0.36778583,-1.17897727)
					(0.35655808,-1.16477273)
					(0.34526893,-1.15056818)
					(0.33391608,-1.13636364)
					(0.32249712,-1.12215909)
					(0.31100950,-1.10795455)
					(0.29945055,-1.09375000)
					(0.28781745,-1.07954545)
					(0.27610723,-1.06534091)
					(0.26431676,-1.05113636)
					(0.25244276,-1.03693182)
					(0.24048174,-1.02272727)
					(0.22843002,-1.00852273)
					(0.21628372,-0.99431818)
					(0.20403871,-0.98011364)
					(0.19169066,-0.96590909)
					(0.17923494,-0.95170455)
					(0.16666667,-0.93750000)
					(0.15398063,-0.92329545)
					(0.14117133,-0.90909091)
					(0.12823288,-0.89488636)
					(0.11515903,-0.88068182)
					(0.10194314,-0.86647727)
					(0.08857809,-0.85227273)
					(0.07505630,-0.83806818)
					(0.06136966,-0.82386364)
					(0.04750951,-0.80965909)
					(0.03346653,-0.79545455)
					(0.01923077,-0.78125000)
					(0.00479150,-0.76704545)
					(-0.00986278,-0.75284091)
					(-0.02474449,-0.73863636)
					(-0.03986700,-0.72443182)
					(-0.05524476,-0.71022727)
					(-0.07089339,-0.69602273)
					(-0.08682984,-0.68181818)
					(-0.10307246,-0.66761364)
					(-0.11964123,-0.65340909)
					(-0.13655789,-0.63920455)
					(-0.15384615,-0.62500000)
					(-0.25000000,-0.62500000)
				} -- cycle;
				
				\addplot[draw=none,fill=figB] coordinates {
					(1.61217949,-3.00000000)
					(2.70000000,-3.00000000)
					(2.70000000,-0.62500000)
					(-0.15384615,-0.62500000)
					(-0.11769424,-0.65506329)
					(-0.08309141,-0.68512658)
					(-0.04984231,-0.71518987)
					(-0.01778313,-0.74525316)
					(0.01322457,-0.77531646)
					(0.04329852,-0.80537975)
					(0.07253954,-0.83544304)
					(0.10103441,-0.86550633)
					(0.12885828,-0.89556962)
					(0.15607653,-0.92563291)
					(0.18274631,-0.95569620)
					(0.20891781,-0.98575949)
					(0.23463525,-1.01582278)
					(0.25993781,-1.04588608)
					(0.28486024,-1.07594937)
					(0.30943356,-1.10601266)
					(0.33368548,-1.13607595)
					(0.35764085,-1.16613924)
					(0.38132203,-1.19620253)
					(0.40474919,-1.22626582)
					(0.42794057,-1.25632911)
					(0.45091269,-1.28639241)
					(0.47368058,-1.31645570)
					(0.49625791,-1.34651899)
					(0.51865718,-1.37658228)
					(0.54088979,-1.40664557)
					(0.56296621,-1.43670886)
					(0.58489605,-1.46677215)
					(0.60668813,-1.49683544)
					(0.62835060,-1.52689873)
					(0.64989095,-1.55696203)
					(0.67131614,-1.58702532)
					(0.69263258,-1.61708861)
					(0.71384623,-1.64715190)
					(0.73496261,-1.67721519)
					(0.75598687,-1.70727848)
					(0.77692379,-1.73734177)
					(0.79777781,-1.76740506)
					(0.81855312,-1.79746835)
					(0.83925357,-1.82753165)
					(0.85988282,-1.85759494)
					(0.88044427,-1.88765823)
					(0.90094109,-1.91772152)
					(0.92137629,-1.94778481)
					(0.94175268,-1.97784810)
					(0.96207289,-2.00791139)
					(0.98233941,-2.03797468)
					(1.00255459,-2.06803797)
					(1.02272063,-2.09810127)
					(1.04283961,-2.12816456)
					(1.06291351,-2.15822785)
					(1.08294417,-2.18829114)
					(1.10293335,-2.21835443)
					(1.12288273,-2.24841772)
					(1.14279387,-2.27848101)
					(1.16266827,-2.30854430)
					(1.18250735,-2.33860759)
					(1.20231245,-2.36867089)
					(1.22208484,-2.39873418)
					(1.24182575,-2.42879747)
					(1.26153633,-2.45886076)
					(1.28121768,-2.48892405)
					(1.30087083,-2.51898734)
					(1.32049680,-2.54905063)
					(1.34009653,-2.57911392)
					(1.35967093,-2.60917722)
					(1.37922087,-2.63924051)
					(1.39874716,-2.66930380)
					(1.41825061,-2.69936709)
					(1.43773196,-2.72943038)
					(1.45719194,-2.75949367)
					(1.47663124,-2.78955696)
					(1.49605052,-2.81962025)
					(1.51545041,-2.84968354)
					(1.53483153,-2.87974684)
					(1.55419444,-2.90981013)
					(1.57353972,-2.93987342)
					(1.59286790,-2.96993671)
					(1.61217949,-3.00000000)
				} -- cycle;
				
				\addplot[draw=black,line width=.65pt,fill=figCFour] coordinates {
					(-1/4,-5/8) (-2/13,-5/8) (-8/37,-85/148)
				} -- cycle;

				\addplot[black,thick,domain=-2:-0.4,samples=240]
				{1/(4*x)-25*x/16};
				\addplot[black,thick,domain=-34/13:-2]
				{-13*x/8-1/4};
				\draw[black,thick] (axis cs:-2,3)--(axis cs:-2,4);
				\draw[black,thick] (axis cs:-2/3,-3)--(axis cs:-2/3,-5/4);
				\draw[black,thick] (axis cs:-3,-5/4)--(axis cs:-2/3,-5/4);
				\addplot[black,thick,domain=-3:-0.625,variable=\t,samples=260]
				({-8*\t/13-2/13+25/(104*\t)},{\t});
				\addplot[black,thick,domain=-1.25:-0.625,variable=\t]
				({2*(\t+0.25)/3},{\t});
				\addplot[black,thick,domain=-0.31:-0.09] {-.625};
				\addplot[black,thick,domain=-0.31:-0.09] {1.5*x-.25};
				\addplot[black,thick,domain=-0.31:-0.09] {-.75-13*x/16};
				
				\draw[black!45,line width=.45pt] (axis cs:0,-3)--(axis cs:0,4);
				\draw[black!45,line width=.45pt] (axis cs:-3,0)--(axis cs:2.70,0);
				
				\draw[black!38,densely dashed]
				(axis cs:-2/3,-3)--(axis cs:-2/3,4);
				\draw[black!38,densely dashed]
				(axis cs:-0.4,-3)--(axis cs:-0.4,4);
				
				\node[font=\normalsize] at (axis cs:-1.22,2.28) {$\mathcal A$};
				\node[font=\normalsize] at (axis cs:1.52,-2.18) {$\mathcal B$};
				\node[font=\small] at (axis cs:-2.15,3.6) {$C_1$};
				\node[font=\normalsize] at (axis cs:-1.65,-2.18) {$C_2$};
				\node[font=\normalsize] at (axis cs:0.30,-1.76) {$C_3$};
				\node[font=\normalsize] at (axis cs:1.42,1.38) {$\mathcal D$};
				
				\draw[
				preaction={draw=white,line width=4.0pt},
				draw=figZoom,line width=1.45pt,
				dash pattern=on 4.5pt off 2.2pt,
				fill=figCFour,fill opacity=.12
				]
				(axis cs:-0.31,-0.70) rectangle (axis cs:-0.09,-0.49);
				
			\end{axis}
			
			\begin{axis}[
				name=zoomplot,
				at={(9.25cm,1.20cm)},
				anchor=south west,
				width=5.10cm,
				height=5.10cm,
				xmin=-0.31,xmax=-0.09,
				ymin=-0.70,ymax=-0.49,
				axis lines=box,
				xtick=\empty,
				ytick=\empty,
				axis line style={figZoom,semithick},
				clip=true,
				enlargelimits=false,
				title={\strut Component $C_4$},
				title style={font=\small,yshift=-1mm},
				]
				
				\addplot[draw=none,fill=figD] coordinates {
					(-0.31,-0.70) (-0.09,-0.70) (-0.09,-0.49) (-0.31,-0.49)
				} -- cycle;
				
				\addplot[draw=none,fill=figCThree] coordinates {
					(-0.30000000,-0.70000000)
					(-0.07025641,-0.70000000)
					(-0.07433900,-0.69600000)
					(-0.07844433,-0.69200000)
					(-0.08257280,-0.68800000)
					(-0.08672485,-0.68400000)
					(-0.09090090,-0.68000000)
					(-0.09510141,-0.67600000)
					(-0.09932682,-0.67200000)
					(-0.10357759,-0.66800000)
					(-0.10785418,-0.66400000)
					(-0.11215705,-0.66000000)
					(-0.11648670,-0.65600000)
					(-0.12084364,-0.65200000)
					(-0.12522839,-0.64800000)
					(-0.12964147,-0.64400000)
					(-0.13408343,-0.64000000)
					(-0.13855482,-0.63600000)
					(-0.14304620,-0.63200000)
					(-0.14755825,-0.62800000)
					(-0.15384615,-0.62500000)
					(-0.25000000,-0.62500000)
				} -- cycle;
				
				\addplot[draw=black,line width=.8pt,fill=figCFour] coordinates {
					(-1/4,-5/8) (-2/13,-5/8) (-8/37,-85/148)
				} -- cycle;
				\addplot[black,thick,domain=-0.31:-0.09] {-.625};
				\addplot[black,thick,domain=-0.31:-0.09] {1.5*x-.25};
				\addplot[black,thick,domain=-0.31:-0.09] {-.75-13*x/16};
				
				\addplot[only marks,mark=*,mark size=3.7pt,
				mark options={draw=white,fill=white}] coordinates {(-0.2,-0.6)};
				\addplot[only marks,mark=*,mark size=2.25pt,
				mark options={draw=figQ,fill=figQ}] coordinates {(-0.2,-0.6)};
				
			\end{axis}
			
			\begin{scope}[shift={(4.64cm,-2.82cm)}]
				\draw[black!55,rounded corners=1.5pt,fill=white]
				(0,0) rectangle (6.55,2.02);
				
				\node[font=\scriptsize,anchor=center] at (3.275,1.70)
				{$\mathcal C=C_1\cup C_2\cup C_3\cup C_4$};
				
				\fill[figA] (.24,1.02) rectangle (.54,1.26);
				\draw[black!50] (.24,1.02) rectangle (.54,1.26);
				\node[font=\scriptsize,anchor=west] at (.66,1.14) {$\mathcal A$};
				
				\fill[figB] (1.76,1.02) rectangle (2.06,1.26);
				\draw[black!50] (1.76,1.02) rectangle (2.06,1.26);
				\node[font=\scriptsize,anchor=west] at (2.18,1.14) {$\mathcal B$};
				
				\fill[figD] (3.28,1.02) rectangle (3.58,1.26);
				\draw[black!50] (3.28,1.02) rectangle (3.58,1.26);
				\node[font=\scriptsize,anchor=west] at (3.70,1.14) {$\mathcal D$};
				
				\fill[figQ] (4.84,1.14) circle[radius=.075];
				\draw[white,line width=.7pt] (4.84,1.14) circle[radius=.075];
				\node[font=\scriptsize,anchor=west] at (5.02,1.14) {$Q$};
				
				\fill[figCOne] (.24,.34) rectangle (.54,.58);
				\draw[black!50] (.24,.34) rectangle (.54,.58);
				\node[font=\scriptsize,anchor=west] at (.66,.46) {$C_1$};
				
				\fill[figCTwo] (1.76,.34) rectangle (2.06,.58);
				\draw[black!50] (1.76,.34) rectangle (2.06,.58);
				\node[font=\scriptsize,anchor=west] at (2.18,.46) {$C_2$};
				
				\fill[figCThree] (3.28,.34) rectangle (3.58,.58);
				\draw[black!50] (3.28,.34) rectangle (3.58,.58);
				\node[font=\scriptsize,anchor=west] at (3.70,.46) {$C_3$};
				
				\fill[figCFour] (4.84,.34) rectangle (5.14,.58);
				\draw[black!50] (4.84,.34) rectangle (5.14,.58);
				\node[font=\scriptsize,anchor=west] at (5.26,.46) {$C_4$};
			\end{scope}
			
		\end{tikzpicture}%
	}
	\caption{Regions $\mathcal A$, $\mathcal B$, $\mathcal C$, and
		$\mathcal D$ for $h=1/4$, with
		$\mathcal C=C_1\cup C_2\cup C_3\cup C_4$.
		The red-framed neighborhood in the left panel is magnified on the right.
		The blue point in the magnification is
		$Q=\left(-1/5,-3/5\right)\in C_4$.}
	\label{fig:regions}
\end{figure}


We can now state the norm formula for $b=1$.

\begin{Lemma}\label{comparisonlemma}
For every $(a,c)\in\mathbb R^2$, $\|(a,1,c)\|_h=\max\{f(a,c),g(a,c)\}$.
Moreover,
\[
\|(a,1,c)\|_h=
\begin{cases}
\displaystyle c-\frac1{4a},
& \text{if } (a,c)\in\mathcal A,\\[10pt]
\displaystyle (h+1)^2\left(a-\frac1{4c}\right),
& \text{if } (a,c)\in\mathcal B,\\[10pt]
\displaystyle
\left|c+h+1+a(h+1)^2-\frac{a(2h+1)+1}{2}\right|
+\frac12|a(2h+1)+1|,
& \text{if } (a,c)\in\mathcal C,\\[12pt]
\displaystyle
\frac12|c+h+1+2a(h+1)^2|+\frac12|c+h+1|,
& \text{if } (a,c)\in\mathcal D.
\end{cases}
\]
\end{Lemma}

\begin{proof}
The identity $\|(a,1,c)\|_h=\max\{f(a,c),g(a,c)\}$ was proved in the previous calculations.
It remains to compare the explicit expressions in~\eqref{f} and~\eqref{g}.

Set $H:=h+1$, $\kappa:=2h+1$ $L:=h^2+(h+1)^2$ and 
\[
q(a,c):=c^2+ch+caL-\frac{H^2}{4}.
\]

Assume first that $(a,c)\in\mathcal R$. Then, $a<0$ and $c>-H/2$, so $(a,c)\notin\mathcal S$ and
\[
f(a,c)=c-\frac1{4a},
\qquad
g(a,c)=\frac12|c+H+2aH^2|+\frac12(c+H).
\]
If $c+H+2aH^2\ge0$, then
\[
f(a,c)-g(a,c)=-\frac{(2aH+1)^2}{4a}\ge0,
\]
and the defining inequality of $\mathcal A$ is automatic. If $c+H+2aH^2<0$, then
\[
f(a,c)-g(a,c)=c+aH^2-\frac1{4a}.
\]
Thus, $f\ge g$ precisely on $\mathcal A$, and $g>f$ on $\mathcal R\setminus\mathcal A$.

Assume next that $(a,c)\in\mathcal S$. Then, $(a,c)\notin\mathcal R$, $c<0$, and $u:=a\kappa+1>0$. Put $s:=2c+\kappa+aL$.
If $s\ge0$, then $f(a,c)=c+H+aH^2$ and
\[
g(a,c)-f(a,c)=-\frac{(2c+H)^2}{4c}>0.
\]
Moreover,
\[
q(a,c)=cs-\left(c+\frac H2\right)^2<0.
\]
If $s<0$, then direct simplification gives
\[
g(a,c)-f(a,c)=\frac{q(a,c)}{c}.
\]
Since $c<0$, one has $g>f$ exactly when $q(a,c)<0$. Therefore, $\mathcal B=\mathcal S\cap\{(a,c):q(a,c)<0\}$, $g>f$ on $\mathcal B$, and $f\ge g$ on $\mathcal S\setminus\mathcal B$. The latter set is contained in $C_3$.

Finally, assume that $(a,c)\notin\mathcal R\cup\mathcal S$. Put $s:=2c+\kappa+aL$, $u:=a\kappa+1$, $v:=c+H+2aH^2$ and $w:=c+H$.
Then, $2f=|s|+|u|$, $2g=|v|+|w|$, $s+u=v+w$.
On the open sign cells, the possible patterns and the corresponding values of $2(f-g)$ are
\[
\renewcommand{\arraystretch}{1.15}
\begin{array}{c|c|c|c|c}
s&u&v&w&2(f-g)\\ \hline
-&-&-&-&0\\
-&-&-&+&-2w\\
-&+&-&-&2u\\
-&+&-&+&2(u-w)\\
-&+&+&+&-2s\\
+&-&-&+&2(s-w)\\
+&-&+&+&-2u\\
+&+&-&+&2v\\
+&+&+&+&0
\end{array}
\]
(the cases in which one of $s,u,v,w$ vanishes follow by continuity). Here, $s-w=c+h+aL$ and $u-w=a\kappa-c-h$.
Combining the table with the conditions $(a,c)\notin\mathcal R\cup\mathcal S$ gives the following assignment. The first row is $C_2$; the third row is contained in $C_3$; the fifth row is $C_4$, together with the boundary $c=-H/2$ assigned to $C_3$; the seventh row is contained in $C_1$. In the fourth row one has $f\ge g$ exactly when $u\ge w$, which gives $C_3$ or $C_4$ according as $c\le-H/2$ or $c>-H/2$. In the sixth row one has $f\ge g$ exactly when $s\ge w$, which gives $C_1$. The remaining rows belong to $\mathcal D$ and satisfy $g\ge f$.

Thus, the set on which the third formula is active is precisely $\mathcal C$, and the complementary set is $\mathcal D$. This proves the formula.
\end{proof}

Consequently, we obtain the formula for arbitrary $b\neq0$.

\begin{Theorem}\label{norm}
Let $b\neq0$. Then,
\[
\|(a,b,c)\|_h=
\begin{cases}
\displaystyle
\operatorname{sgn}(b)\left(c-\frac{b^2}{4a}\right),
& \text{if } \left(\dfrac ab,\dfrac cb\right)\in\mathcal A,\\[12pt]
\displaystyle
\operatorname{sgn}(b)(h+1)^2\left(a-\frac{b^2}{4c}\right),
& \text{if } \left(\dfrac ab,\dfrac cb\right)\in\mathcal B,\\[12pt]
\displaystyle
\left|c+b(h+1)+a(h+1)^2-\frac{a(2h+1)+b}{2}\right|
+\frac12|a(2h+1)+b|,
& \text{if } \left(\dfrac ab,\dfrac cb\right)\in\mathcal C,\\[12pt]
\displaystyle
\frac12|c+b(h+1)+2a(h+1)^2|+\frac12|c+b(h+1)|,
& \text{if } \left(\dfrac ab,\dfrac cb\right)\in\mathcal D.
\end{cases}
\]
\end{Theorem}

\begin{proof}
Apply Lemma~\ref{comparisonlemma} to $(a/b,c/b)$ and multiply by $|b|$ in~\eqref{111}. Each of the four formulas follows by direct simplification.
\end{proof}

\section{Projection of ${\mathsf B}_h$ onto the $ac$-plane}\label{sec:Projection}

For the rest of this paper, let us denote for simplicity
	$$
	a_h := \frac{1}{(h+1)^2}.
	$$
Let us define two functions $\Gamma_h:\left[-a_h,a_h\right]\longrightarrow \mathbb{R}$ and $\Lambda_h:\left[-a_h,a_h\right]\longrightarrow \mathbb{R}$ by
\begin{align*}
	\Gamma_h(a) & := \left[\sqrt{1+h^2 a_h}+h\sqrt{a+a_h}\right]^2,\\
	\Lambda_h(a) & := -\Gamma_h(-a).
\end{align*}
Also, in this section we need to define the sets $A_h$, $B_h$, $C_h$, and $D_h$.
These definitions depend on the parameter $h$, and we have to consider two cases depending on whether $h\leq h_0$ or $h\geq h_0$.

On the one hand, if $h\leq h_0$, then
\begin{align*}
	A_h & := \left\{ (a,c)\in\mathbb{R}^2 :  -a_h \leq a\leq \frac{2h}{h+1},\, \max\{1+h^2a,1+(h^2+h)a\}\leq c\leq\Gamma_h(a) \right\}\\
	&\quad\cup \left\{ (a,c)\in\mathbb{R}^2 : \frac{2h}{h+1}\leq a\leq a_h,\, 1+(h^2+h)a \leq c\leq 2h+1+(h^2+h)a \right\},\\
	B_h & := \left\{ (a,c)\in\mathbb{R}^2 : -a_h\leq a\leq0,\, 1+a(h+1)^2\leq c\leq1+ah^2\right\},\\
	C_h & := \left\{ (a,c)\in\mathbb{R}^2 : -a_h\leq a\leq a_h,\right.\\
	& \qquad \left. \max\{-1+a(h+1)^2,ah(h+1)-2h-1\}\leq c\leq\min\{1+a(h+1)^2,1+ah(h+1)\}\right\},\\
	D_h & := \left\{ (a,c)\in\mathbb{R}^2 : -\frac{2h}{h+1}\leq a\leq a_h,\, \Lambda_h(a)\leq c\leq-1+a(h+1)^2\right\}.
\end{align*}
The sets $A_h$, $B_h$, $C_h$, and $D_h$ have been represented in Figure~\ref{fig:proj_less_h0} for the case $h\leq h_0$.

On the other hand, if $h\geq h_0$, then
\begin{align*}
	A_h & := \left\{ (a,c)\in\mathbb{R}^2 : -a_h\leq a\leq a_h,\, \max\{1+h^2a,1+(h^2+h)a\}\leq c\leq\Gamma_h(a)\right\},\\
	B_h & := \left\{ (a,c)\in\mathbb{R}^2 : -a_h\leq a\leq0,\, 1+a(h+1)^2\leq c\leq1+ah^2\right\},\\
	C_h & := \left\{ (a,c)\in\mathbb{R}^2 : 
	\begin{array}{l}
		\displaystyle -a_h\leq a\leq a_h,\\[3pt]
		\displaystyle -1+a(h+1)^2\leq c\leq \min\{1+a(h+1)^2,1+ah(h+1)\}
	\end{array} 
	\right\},\\
	D_h & := \left\{ (a,c)\in\mathbb{R}^2 : -a_h \leq a \leq a_h,\, \Lambda_h(a) \leq c\leq -1+a(h+1)^2 \right\}.
\end{align*}
The sets $A_h$, $B_h$, $C_h$, and $D_h$ have been represented in Figure~\ref{fig:proj_bigger_h0} for the case $h\geq h_0$.
At $h=h_0$ the two descriptions coincide, since
\[
\pm a_{h_0} = \pm \frac{2h_0}{h_0+1},
\]
and the curved and affine upper boundaries of $A_h$ have the same value at this common endpoint.

\begin{figure}
	\begin{center}
		\includegraphics[width=.6\textwidth]{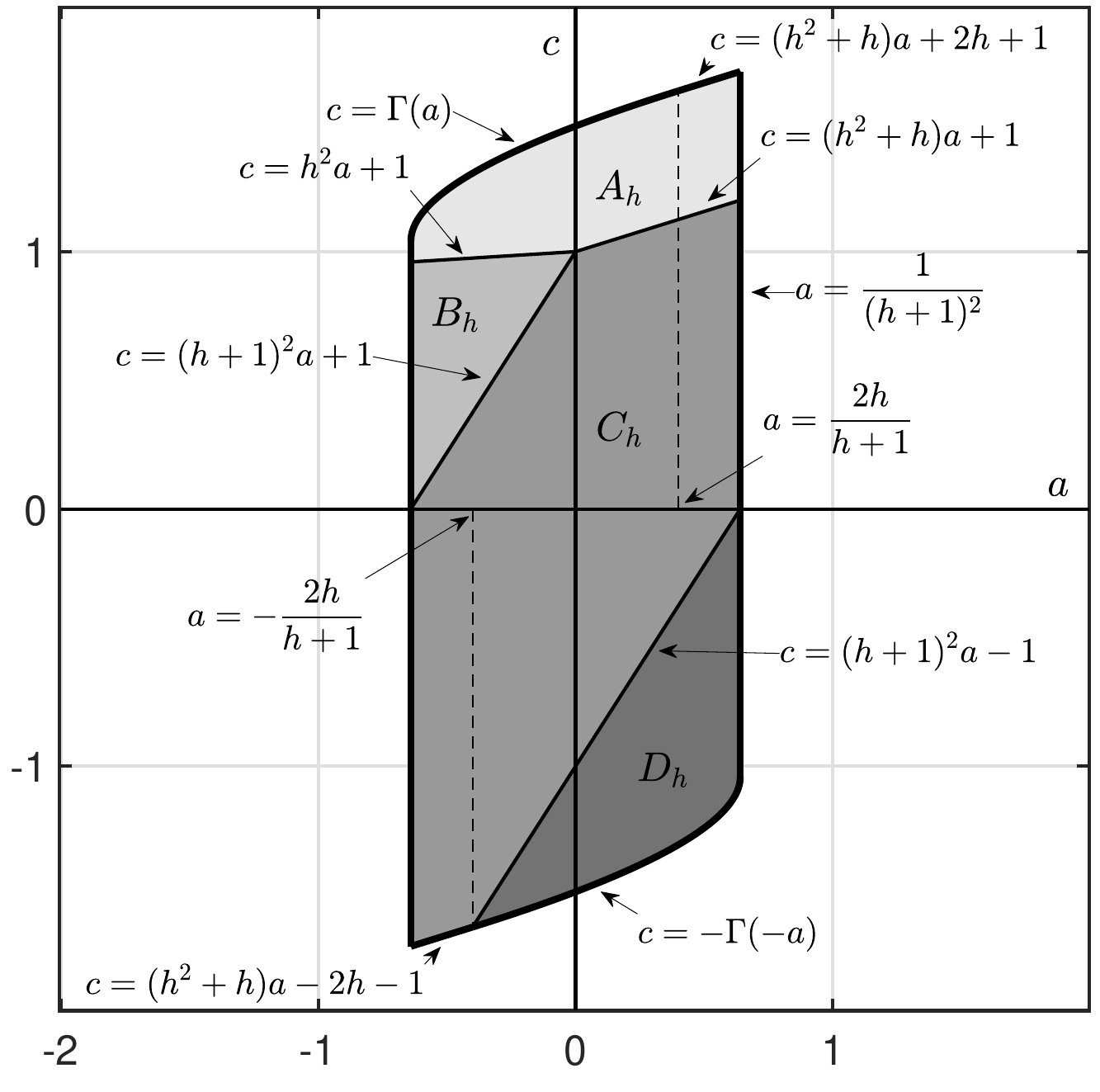}
		\caption{Projection of ${\mathsf B}_h$ onto the $ac$-plane for $h=\frac{1}{4}$.}
		\label{fig:proj_less_h0}
	\end{center}
\end{figure}

\begin{figure}
	\begin{center}
		\includegraphics[width=.6\textwidth]{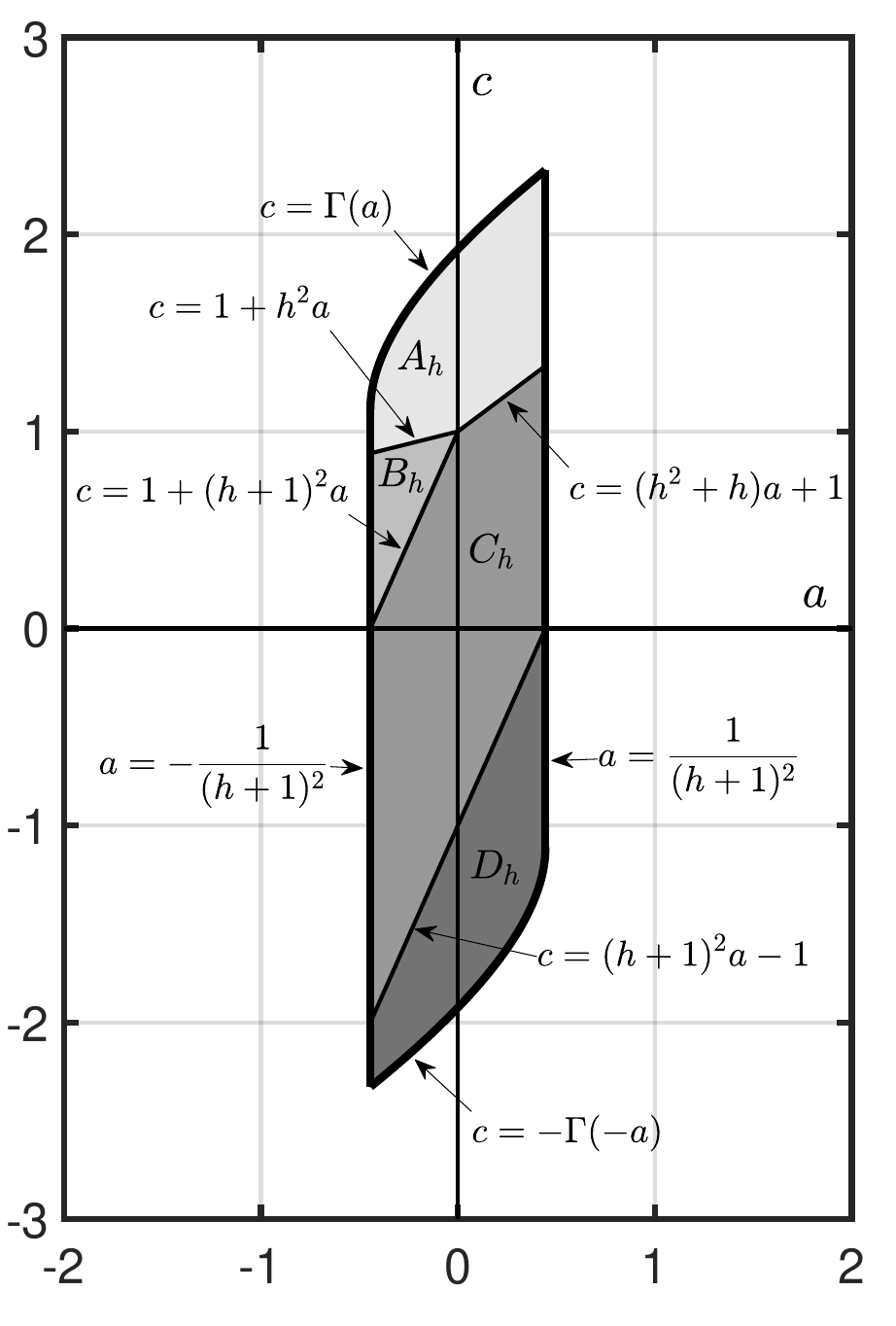}
		\caption{Projection of ${\mathsf B}_h$ onto the $ac$-plane
			for $h=1/2$.}
		\label{fig:proj_bigger_h0}
	\end{center}
\end{figure}

Having all these definitions in mind, we provide below a formula for the projection of ${\mathsf B}_h$ onto the $ac$-plane.

\begin{Theorem}\label{projac}
	If $h>0$, then $\pi_{ac}({\mathsf B}_h)=A_h\cup B_h\cup C_h\cup D_h$.
\end{Theorem}

\begin{proof}
	For $b\in\mathbb R$, put 
		$$
		q_b(x):=ax^2+bx+c \qquad \text{ and } \qquad r_b(y):=a(h+1)^2+b(h+1)y+cy^2.
		$$
	As observed in Section~\ref{sec:formula}, the norm $\|(a,b,c)\|_h$ is attained on $M$.
	Hence, $(a,c)\in\pi_{ac}({\mathsf B}_h)$ if and only if there exists $b\in\mathbb R$ such that $-1\leq q_b(x)\leq1$ for every $x\in[h,h+1]$, and $-1\leq r_b(y)\leq1$ for every $y\in[0,1]$.
	Taking $y=0$ in the second inequality gives $-a_h\leq a\leq a_h$.
	
	The upper inequalities give 
		$$
		b\leq\frac{1-c}{x}-ax
		$$
	for every $x\in[h,h+1]$, and 
		$$
		b\leq\frac{1-a(h+1)^2}{(h+1)y}-\frac{c}{h+1}y
		$$
	for every $y\in(0,1]$. 
	Thus, if
	\begin{align*}
		U_1(a,c) & := \inf_{x\in [h,h+1]} \left(\frac{1-c}{x}-ax\right),\\
		U_2(a,c) & := \inf_{y \in (0,1]} \left(\frac{1-a(h+1)^2}{(h+1)y}-\frac{c}{h+1}y\right),
	\end{align*}
	then the largest admissible value of $b$ is 
		$$
		U(a,c):=\min\{U_1(a,c),U_2(a,c)\}.
		$$
	Similarly, the lower inequalities give $b\geq L(a,c):=-U(-a,-c)$.
	Consequently, $(a,c)\in\pi_{ac}({\mathsf B}_h)$ if and only if $-a_h \leq a\leq a_h$ and $L(a,c)\leq U(a,c)$.
	
	A direct one-variable minimization gives the following four
	possible values of $U(a,c)$:
		$$
		\frac{1-h^2a-c}{h},\qquad 2\sqrt{a(c-1)},\qquad \frac{1-(h+1)^2a-c}{h+1},\qquad2\sqrt{c\left(a-\frac{1}{(h+1)^2}\right)}.
		$$
	The first and third values correspond to the endpoints $x=h$ and $x=h+1$, respectively.
	The second value occurs when the minimum of $U_1$ is attained at an interior critical point of $[h,h+1]$.
	The fourth value is obtained at an interior critical point of $U_2$ whenever that point belongs to $(0,1)$; the limiting boundary cases follow by continuity.
	In particular, when $a=a_h$ and $c\leq 0$, the same formula gives the infimum of $U_2$. If $c<0$, this infimum is approached only as $y\downarrow 0$.
	
	\medskip
	
	We now identify the four resulting regions.
	
	\medskip
	
	First, assume that 
		$$
		U(a,c)=b_h:=\frac{1-h^2a-c}{h}
		$$ 
	The condition that the minimum of $U_1$ is attained at $x=h$ is equivalent to 
		$$
		c\geq\max\{1+h^2a,1+h(h+1)a\}.
		$$
	Moreover, the latter inequality implies $c>0$, so $r_{b_h}$ is convex.
	Since 
		$$
		r_{b_h}(0)=a(h+1)^2\leq 1 \qquad \text{and} \qquad r_{b_h}(1)=q_{b_h}(h+1)\leq 1,
		$$
	it follows that 
		$$
		r_{b_h}(y)\leq (1-y)r_{b_h}(0)+yr_{b_h}(1)\leq 1
		$$
	for every $y\in[0,1]$. 
	Thus, $b_h\leq U_2(a,c)$.
	Put 
		$$
		\tau:=\sqrt{a+a_h},\qquad \rho:=\sqrt{1+h^2 a_h} \qquad \text{and} \qquad  z:=\sqrt c.
		$$ 
	Notice that 
		$$
		q_{b_h}(h)=1 \qquad \text{and} \qquad q_{b_h}(h+1)+1 = \frac{2h+1+ah(h+1)-c}{h}.
		$$
	Moreover, 
		\begin{equation}\label{eq:rbh1}
			r_{b_h}(y)+1 = (zy-(h+1)\tau)^2+\frac{h+1}{h}\left(\rho^2-(z-h\tau)^2\right)y.
		\end{equation}
	We also have 
		$$
		2h+1+ah(h+1)-\Gamma_h(a) = h(\rho-\tau)^2,
		$$
	and 
		$$
		\tau\leq\rho \qquad \text{if and only if} \qquad a\leq\frac{2h}{h+1}.
		$$
	
	Assume first that $a\leq 2 h/(h+1)$.
	Then, $\tau\leq\rho$. 
	By \eqref{eq:rbh1} we have
		$$
		r_{b_h}(y)\geq-1 \text{ for every } y\in[0,1] \qquad \text{if and only if} \qquad c\leq(\rho+h\tau)^2=\Gamma_h(a).
		$$
	Indeed, assume that $c>\Gamma_h(a)$.
	If $\tau>0$, then 
		$$
		z>\rho+h\tau\geq(h+1)\tau,
		$$
	and evaluating $r_{b_h}$ at $y=(h+1)\tau/z \in(0,1]$ gives $r_{b_h}(y)+1<0$. 
	If $\tau=0$, then $a=-a_h$, $r_{b_h}(0)=-1$, and $r_{b_h}'(0)=(\rho^2-c)(h+1)/h<0$, so $r_{b_h}(y)<-1$ for every sufficiently small $y>0$.
	Conversely, assume that $c\leq\Gamma_h(a)$.
	Then, $z\leq\rho+h\tau$, and hence $z-h\tau\leq\rho$.
	We also have $z-h\tau\geq-\rho$.
	Indeed, if $h\tau\leq\rho$, this is immediate. 
	If $h\tau>\rho$, then 
		$$
		1+h^2a-(h\tau-\rho)^2 = 2h\left(\tau\rho-\frac{h}{(h+1)^2}\right)>0,
		$$
	and therefore the lower bound $c\geq1+h^2a$ gives $z\geq h\tau-\rho$.
	Consequently, $|z-h\tau|\leq\rho$, so that 
		$$
		\rho^2-(z-h\tau)^2\geq0.
		$$
	It now follows from \eqref{eq:rbh1} that $r_{b_h}(y)\geq-1$ for every $y\in[0,1]$.
	
	Since $\Gamma_h(a)\leq2h+1+ah(h+1)$, the inequality $c\leq\Gamma_h(a)$ also gives $q_{b_h}(h+1)\geq-1$. 
	For $t\in[0,1]$, the quadratic interpolation identity gives 
		$$
		q_{b_h}(h+t) = (1-t)q_{b_h}(h)+tq_{b_h}(h+1)-at(1-t).
		$$
	Using $q_{b_h}(h)=1$, $q_{b_h}(h+1)\geq-1$ and $a\leq a_h\leq1$, we obtain
		$$
		q_{b_h}(h+t)+1 \geq (1-t)+t(-1)-at(1-t)+1 = (1-t)(2-at)\geq0.
		$$
	Thus, $q_{b_h}(x)\geq-1$ for every $x\in[h,h+1]$.
	
	Now assume that $a\geq 2h/(h+1)$.
	This case can occur only when $h\leq h_0$. 
	Since $h_0<1$, we have $h<1$. 
	For $c=\Gamma_h(a)$, we have $z=\rho+h\tau$, and a direct calculation gives 
		$$
		r_{b_h}'(1)=2z(\rho-\tau)\leq 0.
		$$
	Moreover, 
		$$
		r_{b_h}'(1)=\frac{h+1}{h}(1-h^2a)+\frac{h-1}{h}c,
		$$
	which is decreasing as a function of $c$, since $h<1$. 
	Therefore, $r_{b_h}'(1)\leq0$ whenever 
		$$
		\Gamma_h(a)\leq c\leq2h+1+ah(h+1).
		$$
	Notice also that 
		$$
		c\geq1+h(h+1)a>0.
		$$
	Hence, 
		$$
		r_{b_h}''(y)=2c>0,
		$$
	so $r_{b_h}'$ is increasing on $[0,1]$. 
	Consequently, for every $y\in[0,1]$, 
		$$
		r_{b_h}'(y)\leq r_{b_h}'(1)\leq0,
		$$
	and therefore $r_{b_h}$ is decreasing on $[0,1]$. 
	It follows that $r_{b_h}(y)\geq-1$ for every $y\in[0,1]$ if and only if 
		$$
		r_{b_h}(1)=q_{b_h}(h+1)\geq-1,
		$$
	which is equivalent to 
		$$
		c\leq2h+1+ah(h+1).
		$$
	For $c\leq\Gamma_h(a)$, the sufficiency part of the preceding square-completion argument remains valid without assuming $\tau\leq\rho$.
	Indeed, $z\leq\rho+h\tau$, so $z-h\tau\leq\rho$.
	If $h\tau\leq\rho$, then $z-h\tau\geq-\rho$ is immediate.
	If $h\tau>\rho$, then
	\[
	1+h^2a-(h\tau-\rho)^2
	=2h\left(\tau\rho-\frac{h}{(h+1)^2}\right)>0,
	\]
	and the lower bound $c\geq1+h^2a$ gives $z\geq h\tau-\rho$.
	Thus, $|z-h\tau|\leq\rho$, and the identity for $r(y)+1$ yields $r(y)\geq-1$ on $[0,1]$.
	Consequently, in this range of $a$, the complete condition is 
		$$
		c\leq2h+1+ah(h+1).
		$$
	
	It follows that, if $h\geq h_0$, then $2h/(h+1)\geq a_h$, so the first case gives exactly $A_h$.
	If $0<h\leq h_0$, the same computation gives exactly the two components in the definition of $A_h$.
	
	\medskip
	
	Second, assume that 
		$$
		U(a,c)=b_B:=2\sqrt{a(c-1)},
		$$
	which is the interior minimum of $U_1$. 
	This occurs precisely when 
		$$
		a<0 \qquad \text{and} \qquad 1+(h+1)^2a<c<1+h^2a.
		$$
	Writing
		$$
		\xi:=\sqrt{\frac{1-c}{-a}}\in[h,h+1], 
		$$
	we have 
		$$
		q_{b_B}(x)=1-(-a)(x-\xi)^2.
		$$
	Since $|x-\xi|\leq1$ and $-a\leq a_h \leq1$, it follows that $0\leq q_{b_B}(x)\leq1$ for every $x\in[h,h+1]$.
	Moreover, put 
		$$
		d:=(h+1)\sqrt{-a} \qquad \text{and} \qquad e:=\sqrt{1-c}.
		$$
	Then, $0\leq e\leq d\leq1$, and 
		$$
		r_{b_B}(y)=y^2-(ey-d)^2.
		$$
	For $y\in[0,1]$, we have $|ey-d|\leq d\leq1$, and hence $-1\leq r_{b_B}(y)\leq y^2\leq1$.
	The boundary cases follow by continuity, and the adjacent formulas for $U(a,c)$ agree there. 
	Thus, this case gives exactly $B_h$.
	
	\medskip
	
	Third, assume that 
		$$
		U(a,c)=b_{h+1}:=\frac{1-(h+1)^2a-c}{h+1}.
		$$
	The condition that the minimum of $U_1$ is attained at $x=h+1$, and that $U_2$ is also minimized at $y=1$, gives 
		$$
		c\leq1+h(h+1)a, \qquad c\leq1+(h+1)^2a \qquad \text{and} \qquad c\geq-1+(h+1)^2a.
		$$
	We have $q_{b_{h+1}}(h+1)=r_{b_{h+1}}(1)=1$.
	The remaining lower endpoint condition on the horizontal side is $q_b(h)\geq-1$, which is equivalent to 
		$$
		c\geq ah(h+1)-2h-1.
		$$
	These endpoint inequalities imply the lower inequalities on the whole of both sides. 
	Indeed, for $t\in[0,1]$, 
		$$
		q_{b_{h+1}}(h+t) = (1-t)q_{b_{h+1}}(h)+tq_{b_{h+1}}(h+1)-at(1-t),
		$$
	and the assumptions $q_{b_{h+1}}(h)\geq-1$, $q_{b_{h+1}}(h+1)=1$ and $a\leq a_h\leq1$ give $q_{b_{h+1}}(h+t)\geq-1$.
	Similarly, 
		$$
		r_{b_{h+1}}(y) = (1-y)a(h+1)^2+y-cy(1-y).
		$$
	Since $a(h+1)^2\geq-1$ and $c\leq1+a(h+1)^2$, we have
	\begin{align*}
		r_{b_{h+1}}(y)+1 & \geq (1-y)a(h+1)^2+y-\left(1+a(h+1)^2\right)y(1-y)+1\\
		& = a(h+1)^2(1-y)^2+y^2+1\\
		& \geq -(1-y)^2+y^2+1 =2y \geq0.
	\end{align*}
	Consequently, this case gives the conditions that appear in the definition of $C_h$ for $h\leq h_0$.
	If $h\geq h_0$, the additional lower bound $c\geq ah(h+1)-2h-1$ is redundant. 
	Indeed, note that 
		$$
		ah(h+1)-2h-1 -(-1+a(h+1)^2) = -a(h+1)-2h,
		$$
	and this quantity is nonpositive for $a\geq-a_h$ precisely when $2h(h+1)\geq1$, that is, when $h\geq h_0$.
	Thus, this third case gives exactly $C_h$ in both parameter regimes.
	
	\medskip
	
	Finally, assume that 
		$$
		U(a,c)=b_D:=2\sqrt{c\left(a-\frac{1}{(h+1)^2}\right)},
		$$
	which is the interior minimum of $U_2$. 
	Put 
		$$
		\psi:=\sqrt{a_h-a}, \qquad  z:=\sqrt{-c} \qquad \text{and} \qquad \rho:=\sqrt{1+h^2 a_h}.
		$$
	Then, $b_D=2z\psi$, and the condition $c\leq-1+(h+1)^2a$ is equivalent to $z\geq(h+1)\psi$.
	For this value of $b_D$, we have 
		$$
		q_{b_D}(x) = \frac{x^2}{(h+1)^2}-(z-x\psi)^2.
		$$
	Since $z\geq(h+1)\psi$, we have $z-x\psi\geq0$ for every $x\in[h,h+1]$. 
	Therefore, 
		$$
		\frac{d}{dx}(q_{b_D}(x)+1) = \frac{2x}{(h+1)^2}+2\psi(z-x\psi)>0.
		$$
	Thus, $q_{b_D}(x)+1$ is increasing on $[h,h+1]$.
	Therefore the lower inequality on the horizontal side is equivalent to $q_{b_D}(h)\geq-1$. 
	This condition is equivalent to $z\leq\rho+h\psi$ or, equivalently, 
		$$
		c\geq-(\rho+h\psi)^2=\Lambda_h(a).
		$$
	The upper inequality on the horizontal side follows directly from 
		$$
		q_{b_D}(x)\leq\frac{x^2}{(h+1)^2}\leq1.
		$$
	
	On the vertical side, 
		$$
		r_{b_D}(y) = 1-(zy-(h+1)\psi)^2.
		$$
	Since $y\longmapsto zy-(h+1)\psi$ is affine on $[0,1]$, it is enough to check its values at the endpoints. 
	At $y=0$, $|(h+1)\psi|\leq\sqrt2$, since $a\geq-a_h$. 
	Moreover, the chain of inequalities 
		$$
		(h+1)\psi\leq z\leq\rho+h\psi
		$$
	implies $\psi\leq\rho$, and hence
		$$
		|z-(h+1)\psi| = z-(h+1)\psi \leq\rho-\psi \leq\rho<\sqrt2.
		$$
	It follows that $|r_{b_D}(y)|\leq1$ for every $y\in[0,1]$.
	
	Thus, the conditions in this fourth case are 
		$$
		\Lambda_h(a)\leq c\leq-1+(h+1)^2a.
		$$
	This interval is nonempty if and only if $\psi\leq\rho$, which is equivalent to $a\geq-2h/(h+1)$.
	If $h\geq h_0$, then $-2h/(h+1) \leq -a_h$, so this restriction is redundant with $a\geq -a_h$. 
	If $0<h\leq h_0$, it gives the lower restriction $a\geq-2h/(h+1)$.
	Therefore, this fourth case gives exactly $D_h$ in both parameter regimes.
	
	\medskip
	
	On every common boundary of two of the regions $A_h,B_h,C_h,D_h$, the corresponding formulas for $U(a,c)$ agree since they are adjacent endpoint or critical-point formulas for the same one-variable minimum.
	Thus, the overlaps cause no ambiguity.

	We have shown that $L(a,c)\leq U(a,c)$ for $-a_h \leq a\leq a_h$ implies $(a,c)\in A_h\cup B_h\cup C_h\cup D_h$.
	Hence, $\pi_{ac}({\mathsf B}_h) \subseteq A_h\cup B_h\cup C_h\cup D_h$.
	
	Conversely, if $(a,c)\in A_h\cup B_h\cup C_h\cup D_h$, the preceding computations give $L(a,c)\leq U(a,c)$. 
	Fix $b\in[L(a,c),U(a,c)]$.
	Then, all the inequalities defining the two side restrictions are satisfied, and therefore $\|(a,b,c)\|_h\leq1$.
	Thus, $(a,c)\in\pi_{ac}({\mathsf B}_h)$.
	This completes the proof.
\end{proof}


\section{Parametrization of ${\mathsf S}_h$}\label{sec:Parametriztion}

\begin{Theorem}\label{parametrization}
	For every $h>0$, we define the functions $F_h,G_h:\pi_{ac}({\mathsf B}_h)\longrightarrow{\mathbb R}$ by
		$$
		F_h(a,c) :=
		\begin{cases}
		\dfrac{1-h^2a-c}{h}, & \text{if } (a,c)\in A_h,\\		
		2\sqrt{a(c-1)}, & \text{if } (a,c)\in B_h,\\
		\dfrac{1-(h+1)^2a-c}{h+1}, & \text{if } (a,c)\in C_h,\\
		2\sqrt{c\left[a-\frac{1}{(h+1)^2}\right]}, & \text{if } (a,c)\in D_h,
		\end{cases}
		$$
	and $G_h(a,c) := -F_h(-a,-c)$.
	Then, ${\mathsf S}_h=\graph(F_h)\cup \graph(G_h) \,\cup \, T_h^+ \, \cup \, T_h^-$, where 
		\begin{align*}
			T_h^+ := \left \{\left (\frac{1}{(h+1)^2},b,c \right ) : c\in I_h,\, G_h\left(\frac{1}{(h+1)^2},c\right)\leq b\leq F_h\left(\frac{1}{(h+1)^2},c\right) \right \},
		\end{align*}
	$T_h^- := -T_h^+$, and
		$$
		I_h:=\begin{cases}
		\left[-\dfrac{2h^2+2h+1}{(h+1)^2},\dfrac{2h^2+4h+1}{h+1}\right], & \text{if } h< h_0,\\[13pt]
		\left[-\Gamma_h\left(-\dfrac{1}{(h+1)^2}\right),\Gamma_h\left(\dfrac{1}{(h+1)^2}\right)\right], & \text{if } h\ge h_0.
		\end{cases}
		$$
\end{Theorem}

\begin{proof}
By Theorem~\ref{projac}, for each $(a,c)\in\pi_{ac}({\mathsf B}_h)$ the admissible values of $b$ form the interval
	$$
	[L(a,c),U(a,c)].
	$$
The four formulas displayed in the definition of $F_h$ are exactly the four possible values of $U(a,c)$ obtained in the proof of Theorem~\ref{projac}. 
Hence,  
	$$
	F_h(a,c)=U(a,c)
	$$
for all $(a,c)\in\pi_{ac}({\mathsf B}_h)$.
The formulas agree on the common boundaries of $A_h,B_h,C_h,D_h$, as observed at the end of the proof of Theorem~\ref{projac}; thus, $F_h$ is well defined. 
By symmetry, 
	$$
	G_h(a,c)=-U(-a,-c)=L(a,c).
	$$
Consequently,
\[
{\mathsf B}_h=\bigl\{(a,b,c)\in\mathbb R^3:(a,c)\in\pi_{ac}({\mathsf B}_h),\ G_h(a,c)\leq b\leq F_h(a,c)\bigr\}.
\]

The functions $U$ and $L$ are continuous on the projection. Assume that
\[
|a|<a_h
\qquad\text{and}\qquad
L(a,c)<U(a,c).
\]
Both inequalities remain strict in a neighborhood of $(a,c)$, and therefore $(a,c)$ is an interior point of $\pi_{ac}({\mathsf B}_h)$. It follows that, at every boundary point of the projection with $|a|<a_h$, one has $L(a,c)=U(a,c)$. Hence, over those boundary arcs, the fibre is a singleton and is already contained in $\graph(F_h)\cap\graph(G_h)$.

It remains to examine the two supporting lines
\[
a=\pm a_h.
\]
Reading the corresponding vertical sections directly from the definitions of $A_h,C_h,D_h$ gives, at $a=a_h$,
\[
c\in I_h.
\]
For $0<h<h_0$, this interval is
\[
\left[-\frac{2h^2+2h+1}{(h+1)^2},\frac{2h^2+4h+1}{h+1}\right],
\]
whereas, for $h\geq h_0$, it is
\[
\left[-\Gamma_h\left(-a_h\right),
\Gamma_h\left(a_h\right)\right].
\]
The vertical boundary piece over this section is precisely $T_h^+$, and the centrally symmetric piece over $a=-a_h$ is $T_h^-=-T_h^+$.

The upper and lower endpoints of every remaining fibre belong to $\graph(F_h)$ and $\graph(G_h)$, respectively. Therefore,
\[
{\mathsf S}_h=\graph(F_h)\cup\graph(G_h)\cup T_h^+\cup T_h^-.
\]
\end{proof}


\section{Extreme points of ${\mathsf B}_h$}\label{sec:ExtremePoints}

\begin{Theorem}
	${}$
	\begin{enumerate}
		\item[\rm{(a)}] If $h\geq h_0$, then $\ext({\mathsf B}_h)$ is given by the following polynomials
		$$
		\pm \left(t,\frac{1-h^2t-\Gamma_h(t)}{h},\Gamma_h(t)\right),\quad\text{$t\in \left [-\frac{1}{(h+1)^2}, \frac{1}{(h+1)^2}\right ]$},
		$$
		$$
		\pm\left(-\frac{1}{(h+1)^2},\frac{2\sqrt{1-s}}{h+1},s\right),\quad\text{$s\in \left[0,\frac{2h+1}{(h+1)^2}\right]$},
		$$
		$$
		\pm\left(\frac{1}{(h+1)^2},-\frac{2\sqrt{2s}}{h+1},s\right),\quad\text{$s\in \left[2,\left(\frac{\sqrt{2h^2+2h+1}+\sqrt{2}h}{h+1}\right)^2\right]$},
		$$
		$$
		\pm (0,0,1),\quad \pm\left(\frac{1}{(h+1)^2},0,0\right),\quad \pm\left(\frac{1}{(h+1)^2},-\frac{2h+1}{(h+1)^2},\frac{2h+1}{h+1}\right).
		$$
		
		\vspace{3pt}
		
		\item[\rm{(b)}] If $0<h< h_0$, then $\ext({\mathsf B}_h)$ is given by the following polynomials
		$$
		\pm \left(t,\frac{1-h^2t-\Gamma_h(t)}{h},\Gamma_h(t)\right),\quad\text{$t\in \left[-\frac{1}{(h+1)^2},\frac{2h}{h+1}\right]$},
		$$
		$$
		\pm\left(-\frac{1}{(h+1)^2},\frac{2\sqrt{1-s}}{h+1},s\right),\quad\text{$s\in \left[0,\frac{2h+1}{(h+1)^2}\right]$},
		$$
		$$
		\pm (0,0,1),\quad \pm\left(\frac{1}{(h+1)^2},-\frac{2h^2+6h+3}{(h+1)^2},\frac{2h^2+4h+1}{h+1}\right),
		$$
		
		$$
		 \pm\left(\frac{1}{(h+1)^2},0,0\right),\quad \pm\left(\frac{1}{(h+1)^2},-\frac{2h+1}{(h+1)^2},\frac{2h+1}{h+1}\right).
		$$
	\end{enumerate}
\end{Theorem}

\begin{proof}
	We first identify all possible extreme points. The pieces $\graph(F_h|_{A_h})$ and $\graph(F_h|_{C_h})$ are contained in affine planes and are therefore ruled. The piece over $B_h$ is ruled by the segments
	\[
	L_\lambda:=\{(a,\lambda a+1):-a_h\leq a\leq0\},
	\qquad h^2\leq\lambda\leq(h+1)^2
	\]
	since
	\[
	B_h=\bigcup_{h^2\leq\lambda\leq(h+1)^2}L_\lambda
	\quad\text{and}\quad
	F_h(a,\lambda a+1)=-2\sqrt\lambda\,a.
	\]
	This is the family depicted in Figure~\ref{fig:LLambda}.
	\begin{figure}
	\begin{center}
	\includegraphics[width=.4\textwidth]{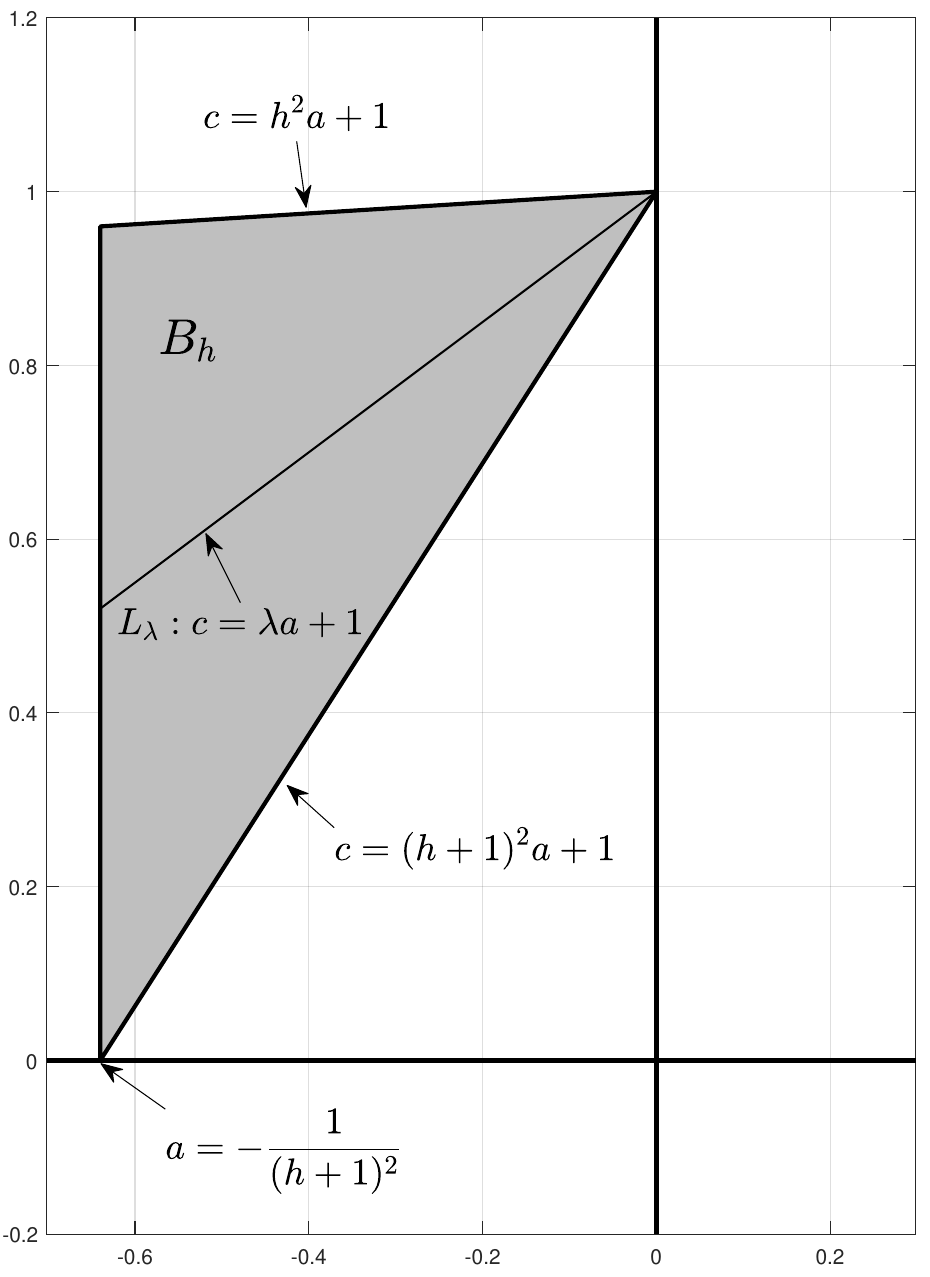}
	\caption{Segment $L_\lambda$ with $h=1/4$ and $\lambda=3/4$.}\label{fig:LLambda}
	\end{center}
	\end{figure}
	
	The piece over $D_h$ is ruled in the same way. For $\lambda\geq(h+1)^2$, set
	\[
	a_\lambda:=a_h-
	\frac{1+h^2 a_h}{\left(\sqrt\lambda-h\right)^2}
	\]
	and
	\[
	R_\lambda:=\left\{\left(a,\lambda(a-a_h)\right):
	\max\{-a_h,a_\lambda\}\leq a\leq a_h\right\}.
	\]
	Then,
	\[
	D_h=\bigcup_{\lambda\geq(h+1)^2}R_\lambda,
	\qquad
	F_h\bigl(a,\lambda(a-a_h)\bigr)
	=2\sqrt\lambda\,(a_h-a).
	\]
	For $0<h<h_0$ one has $a_\lambda>-a_h$, so this reduces to the parametrization shown in Figure~\ref{fig:R_Lambda}.
	\begin{figure}
	\begin{center}
	\includegraphics[width=.4\textwidth]{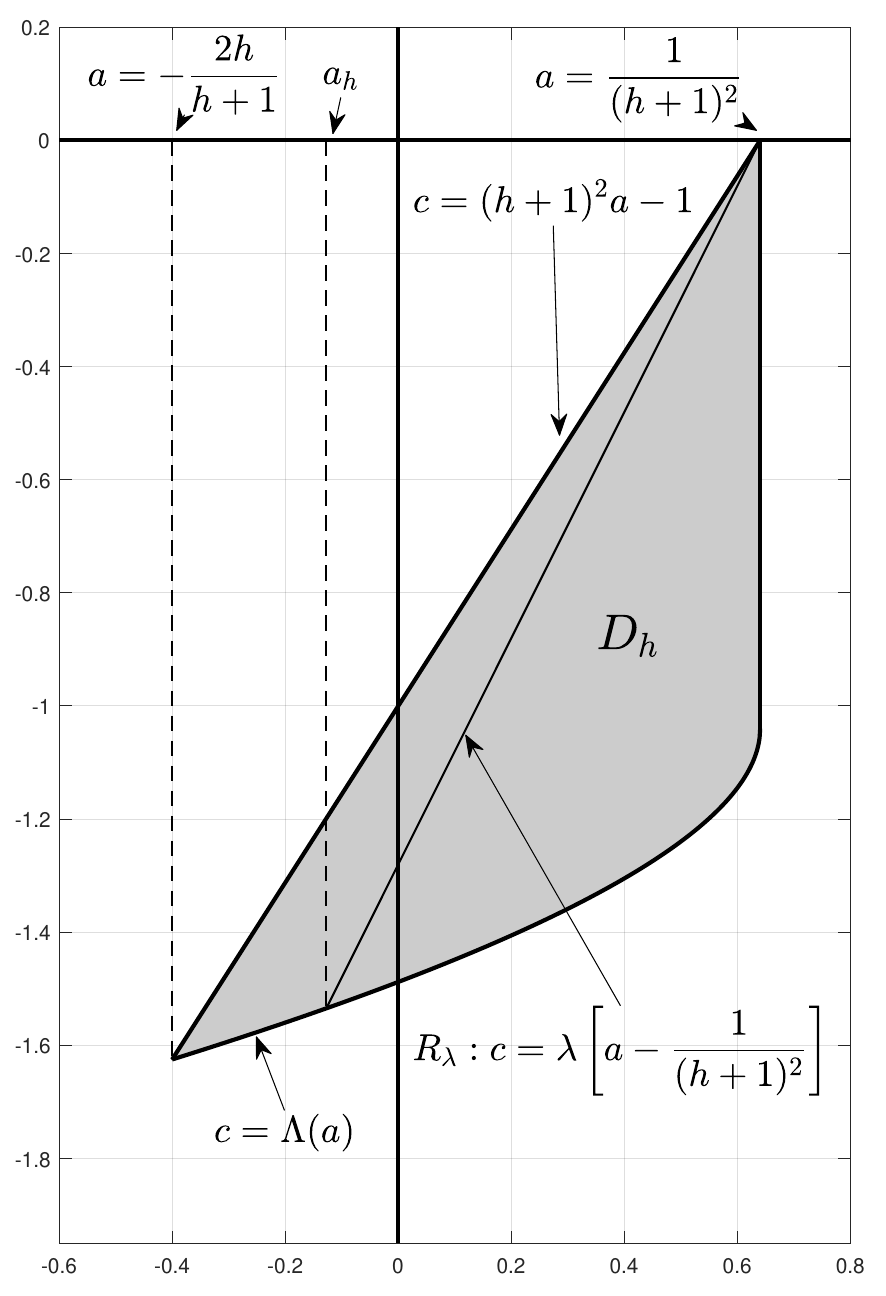}
	\caption{Segment $R_\lambda$ with $h=1/4$ and $\lambda=2$.}\label{fig:R_Lambda}
	\end{center}
	\end{figure}
	By central symmetry, the corresponding pieces of the lower graph are ruled as well. No point in the relative interior of a nontrivial ruling segment can be extreme. Thus, all candidates occur on the boundary curves and intersections of these ruled pieces or on the two vertical faces $T_h^\pm$.
	
	We now treat the two parameter regimes.
	
	\medskip
	\noindent\textbf{Case 1: $0<h<h_0$.}
	Define
	\[
	\Upsilon_h:=\left\{\left(t,\frac{1-h^2t-\Gamma_h(t)}h,\Gamma_h(t)\right):
	-a_h\leq t\leq\frac{2h}{h+1}\right\}
	\]
	and
	\[
	\Phi_h:=\left\{\left(-a_h,\frac{2\sqrt{1-s}}{h+1},s\right):
	0\leq s\leq\frac{2h+1}{(h+1)^2}\right\}.
	\]
	Introduce the points
	\begin{align*}
	p_1&:=\left(-a_h,\frac{2h}{(h+1)^2},1-\frac{h^2}{(h+1)^2}\right),
	& p_2&:=\left(-a_h,0,1+\frac{h^2}{(h+1)^2}\right),\\
	p_3&:=\left(a_h,-\frac{2h^2+6h+3}{(h+1)^2},\frac{2h^2+4h+1}{h+1}\right),
	& p_4&:=\left(a_h,-\frac{2h+1}{(h+1)^2},\frac{2h+1}{h+1}\right),\\
	p_5&:=(0,0,1),
	& p_6&:=\left(-a_h,\frac2{h+1},0\right),\\
	p_7&:=\left(-\frac{2h}{h+1},\frac{2(2h^2+2h+1)}{h+1},-(2h^2+2h+1)\right),
	& p_8&:=\left(a_h,0,0\right).
	\end{align*}
	Direct substitution in the formulas of Theorem~\ref{parametrization} gives
	\begin{align*}
	\graph(F_h|_{A_h})\cap T_h^-&=[p_1,p_2],
	&\graph(F_h|_{A_h})\cap T_h^+&=[p_3,p_4],\\
	\graph(F_h|_{A_h})\cap\graph(F_h|_{C_h})&=[p_4,p_5],
	&\graph(F_h|_{A_h})\cap\graph(F_h|_{B_h})&=[p_5,p_1],\\
	\graph(F_h|_{B_h})\cap\graph(F_h|_{C_h})&=[p_6,p_5],
	&\graph(F_h|_{B_h})\cap T_h^-&=\Phi_h,\\
	\graph(F_h|_{C_h})\cap T_h^-&=[-p_3,p_6],
	&\graph(F_h|_{C_h})\cap T_h^+&=[p_8,p_4],\\
	\graph(F_h|_{C_h})\cap\graph(F_h|_{D_h})&=[p_7,p_8].
	\end{align*}
	Moreover,
	\[
	\begin{aligned}
	& \graph(F_h|_{A_h})\cap\graph(G_h|_{A_h})\\
	& \qquad = \Upsilon_h\cup \left\{\left(t,-2-(2h+1)t,2h+1+h(h+1)t\right):\frac{2h}{h+1}\leq t\leq a_h\right\}.
	\end{aligned}
	\]
	where the second component is a line segment, and
	\[
	\graph(F_h|_{D_h})\cap\graph(G_h|_{-A_h})=-\Upsilon_h.
	\]
	Finally, $\graph(F_h|_{D_h})\cap T_h^-=\varnothing$, while
	$\graph(F_h|_{D_h})\cap T_h^+$ is a segment whose endpoints are $p_8$ and $-p_2$.
	It follows that the only candidates are
	\[
	\pm\Upsilon_h,\qquad \pm\Phi_h,\qquad
	\pm p_3,\quad\pm p_4,\quad\pm p_5,\quad\pm p_8.
	\]
	
	\medskip
	\noindent\textbf{Case 2: $h\geq h_0$.}
	In this case define
	\[
	\Upsilon_h:=\left\{\left(t,\frac{1-h^2t-\Gamma_h(t)}h,\Gamma_h(t)\right):
	-a_0\leq t\leq a_0\right\},
	\]
	retain the preceding definition of $\Phi_h$, and set
	\[
	\Psi_h:=\left\{\left(a_h,-\frac{2\sqrt{2s}}{h+1},s\right):
	2\leq s\leq\Gamma_h(a_h)\right\}.
	\]
	Notice that
	\[
	\Gamma_h(a_h)=
	\left(\frac{\sqrt{2h^2+2h+1}+\sqrt2\,h}{h+1}\right)^2.
	\]
	Let $u_+$ denote the right endpoint of $\Upsilon_h$ and put
	\[
	v_-:=\left(-a_h,\frac4{h+1},-2\right).
	\]
	With $p_1,p_2,p_4,p_5,p_6,p_8$ as above, direct substitution gives
	\begin{align*}
	\graph(F_h|_{A_h})\cap T_h^-&=[p_1,p_2],
	&\graph(F_h|_{A_h})\cap T_h^+&=[u_+,p_4],\\
	\graph(F_h|_{A_h})\cap\graph(F_h|_{C_h})&=[p_4,p_5],
	&\graph(F_h|_{A_h})\cap\graph(F_h|_{B_h})&=[p_5,p_1],\\
	\graph(F_h|_{B_h})\cap\graph(F_h|_{C_h})&=[p_6,p_5],
	&\graph(F_h|_{B_h})\cap T_h^-&=\Phi_h,\\
	\graph(F_h|_{C_h})\cap T_h^-&=[v_-,p_6],
	&\graph(F_h|_{C_h})\cap T_h^+&=[p_8,p_4],\\
	\graph(F_h|_{C_h})\cap\graph(F_h|_{D_h})&=[v_-,p_8],
	&\graph(F_h|_{D_h})\cap T_h^-&=-\Psi_h.
	\end{align*}
	Furthermore,
	\[
	\graph(F_h|_{A_h})\cap\graph(G_h|_{A_h})=\Upsilon_h,
	\qquad
	\graph(F_h|_{D_h})\cap\graph(G_h|_{-A_h})=-\Upsilon_h,
	\]
	and $\graph(F_h|_{D_h})\cap T_h^+$ is the segment $[-p_2,p_8]$.
	The lower endpoint of $\Psi_h$ is $-v_-$, and its upper endpoint is $u_+$. Therefore, the complete list of candidates is
	\[
	\pm\Upsilon_h,\qquad \pm\Phi_h,\qquad \pm\Psi_h,
	\qquad \pm p_4,\quad\pm p_5,\quad\pm p_8.
	\]
	At $h=h_0$ the interval defining $\Psi_h$ degenerates at its lower endpoint, so the two descriptions fit continuously.
	
	It remains to prove that all the listed candidates are extreme.
	
	First consider $\Upsilon_h$, and let $\gamma_h$ denote its projection onto the $ac$-plane. On the relevant interval,
	\[
	\Gamma_h''(t)
	=-\frac{h\sqrt{1+h^2a_h}}
	{2\left(t+a_h\right)^{3/2}}<0.
	\]
	Thus, every interior point of $\gamma_h$ is extreme in $\pi_{ac}({\mathsf B}_h)$. The left endpoint is an endpoint of the exposed face $a=-a_h$. If $h\geq h_0$, the right endpoint is an endpoint of the exposed face $a=a_h$; if $0<h<h_0$, it is the endpoint of the adjacent maximal affine boundary segment. Hence, both endpoints are extreme as well. The fibre over every point of $\gamma_h$ is a singleton because
	\[
	F_h(t,\Gamma_h(t))=G_h(t,\Gamma_h(t)).
	\]
	If a point of $\Upsilon_h$ were the midpoint of two points of ${\mathsf B}_h$, their projections would both equal the same extreme point of $\gamma_h$, and the singleton-fibre property would force the two points to coincide. Hence, every point of $\Upsilon_h$ is extreme.
	
	Next, the hyperplane $a=-a_h$ supports ${\mathsf B}_h$. In the face ${\mathsf B}_h\cap\{a=-a_h\}$, the curve $\Phi_h$ is the upper boundary
	\[
	b=\frac{2\sqrt{1-c}}{h+1},
	\]
	and the face lies below this graph. Since this function is strictly concave, every point of $\Phi_h$ is extreme in the face and hence in ${\mathsf B}_h$.
	
	When $h\geq h_0$, the hyperplane $a=a_h$ is also supporting. In the face ${\mathsf B}_h\cap\{a=a_h\}$, the curve $\Psi_h$ is the lower boundary
	\[
	b=-\frac{2\sqrt{2c}}{h+1},
	\]
	and the face lies above this graph. The displayed function is strictly convex on $c>0$, so every point of $\Psi_h$ is extreme in this face and therefore in ${\mathsf B}_h$.
	
	Finally, for $(x,y)\in\square_h$, write
	\[
	\ell_{x,y}(a,b,c):=ax^2+bxy+cy^2.
	\]
	Since $|\ell_{x,y}|\leq1$ on ${\mathsf B}_h$, the equations $\ell_{x,y}=\pm1$ define supporting hyperplanes.
	For $0<h<h_0$, the point $p_3$ satisfies
	\[
	\ell_{h,1}(p_3)=1,\qquad
	\ell_{h+1,0}(p_3)=1,\qquad
	\ell_{h+1,1}(p_3)=-1.
	\]
	The point $p_4$ satisfies the same three equations with the last value equal to $1$. The corresponding normal vectors
	\[
	(h^2,h,1),\qquad ((h+1)^2,0,0),\qquad ((h+1)^2,h+1,1)
	\]
	have determinant $(h+1)^2\neq0$. Hence, $p_3$ and $p_4$ are exposed whenever they occur in the list.
	
	For $p_5=(0,0,1)$, one has $\ell_{x,1}(p_5)=1$ for every $x\in[h,h+1]$. Choosing three distinct values of $x$ gives three linearly independent Vandermonde normal vectors, so $p_5$ is exposed. For $p_8=(a_h,0,0)$, one has $\ell_{h+1,y}(p_8)=1$ for every $y\in[0,1]$. Taking $y=0,1/2,1$ gives three normal vectors with determinant $(h+1)^3/4\neq0$, so $p_8$ is exposed.
	
	Central symmetry now gives the extremality of all negatives. The preceding ruling analysis proves that no further extreme points exist, and the two lists in the statement follow.
\end{proof}

\begin{figure}
	\begin{center}
		\includegraphics[width=.4\textwidth]{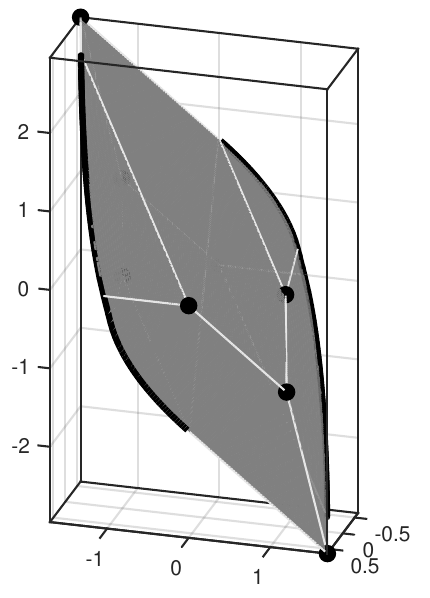}
		\caption{Representation of ${\mathsf S}_h$  for $h=1/4$. The curves of extreme points are depicted with thick lines whereas isolated extreme points appear as small black spheres.}
	\end{center}
\end{figure}

\begin{figure}
	\begin{center}
		\includegraphics[width=.4\textwidth]{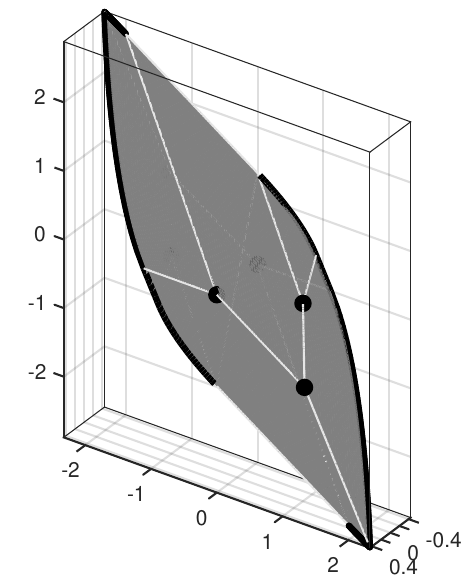}
		\caption{Representation of ${\mathsf S}_h$  for $h=1/2$. The curves of extreme points are depicted with thick lines whereas isolated extreme points appear as small black spheres.}
	\end{center}
\end{figure}

%
%

\end{document}